\documentclass[12pt,reqno]{amsart}
\usepackage{amscd,amssymb,amsthm,verbatim,slashed}
\newcommand{\C}{{\mathbb C}}

\newcommand{\ch}{\operatorname{ch}}

\newcommand{\dvol}{\operatorname{dvol}}
\newcommand{\End}{\operatorname{End}}

\newcommand{\HH}{\operatorname{H}}

\newcommand{\Hom}{\operatorname{Hom}}

\newcommand{\Id}{\operatorname{Id}}

\newcommand{\Image}{\operatorname{Im}}

\newcommand{\Ker}{\operatorname{Ker}}

\newcommand{\Pin}{\operatorname{Pin}}

\newcommand{\Q}{{\mathbb Q}}
\newcommand{\R}{{\mathbb R}}

\newcommand{\Ric}{\operatorname{Ric}}

\newcommand{\SO}{\operatorname{SO}}

\newcommand{\Spin}{\operatorname{Spin}}

\newcommand{\Tr}{\operatorname{Tr}}

\newcommand{\Z}{{\mathbb Z}}

\numberwithin{equation}{section}
\theoremstyle{plain}
\newtheorem{definition}[equation]{Definition}

\newtheorem{lemma}[equation]{Lemma}
\newtheorem{theorem}[equation]{Theorem}
\newtheorem{proposition}[equation]{Proposition}
\newtheorem{corollary}[equation]{Corollary}

\theoremstyle{remark}

\newtheorem{remark}[equation]{Remark}

\usepackage{graphicx}
\input{amssym.def}
\input{amssym}
\input{epsf}
\usepackage{a4wide}

\begin{document}

\pagestyle{plain}

\title[A spinorial quasilocal mass]
      {A spinorial quasilocal mass}

\author{John Lott\\
Department of Mathematics\\
University of California, Berkeley\\
Berkeley, CA  94720-3840\\
USA
}

\footnote{lott@berkeley.edu}

\begin{abstract}
We define a quasilocal energy of a compact manifold-with-boundary, relative to a background manifold.
The construction uses spinors on one manifold and the pullback of dual spinors from the other
manifold. We prove positivity results for the quasilocal energy, in both the Riemannian and Lorentzian settings. 
\end{abstract}

\date{December 3, 2023}

\maketitle

\section{Introduction} \label{sect1}

The mass of an asymptotically flat initial data set is a quantity computed from the asymptotic geometry.  It has the physical interpretation of measuring the gravitational energy within the
space.  The positive mass theorem says that under an appropriate positivity assumption on the curvature, 
the mass is positive \cite{Schoen-Yau (1981), Witten (1981)}.

It is an old problem to localize the measurement of gravitational energy to a compact region, i.e. to
give a good notion of a quasilocal mass. Some survey articles are
\cite{Szabados (2009),Wang (2015)}. 
In most of this paper we will refer more precisely to a quasilocal energy rather than a quasilocal mass.

In view of Witten's proof of the positive energy theorem for spin manifolds \cite{Witten (1981)},
an attractive approach to define a quasilocal energy
is to use spinors. This idea has been
considered by various people, notably Dougan-Mason \cite{Dougan-Mason (1991)} and Zhang
\cite{Zhang (2008),Zhang (2009)}.  There are two basic issues with such an approach. First,
the quasilocal energy is always defined relative to some background space.  For example, in the positive
energy theorem, the background space is Euclidean space.  It is not clear how to incorporate the
background space into the definition of the quasilocal energy.  The second issue is what boundary
conditions to impose on the spinor fields.  Witten's proof used spinor fields that asymptotically
approach a constant spinor with respect to the model flat space.  It is not evident what the analog 
should be for a compact manifold-with-boundary.

We show that both of these issues can be handled by the technique of
pulling back (dual) spinors from the background
space.  This technique has some history in the mathematical literature and
has been applied to prove sharp comparison results about scalar curvature by
Llarull \cite{Llarull (1996),Llarull (1998)}, Goette-Semmelmann \cite{Goette-Semmelmann (2002)},
the author \cite{Lott (2021)} and Wang-Xie-Yu \cite{Wang-Xie-Yu (2021)}.   One advantage
of the spinorial approach to the quasilocal energy is that the nonnegativity of the energy, under
some curvature conditions, falls out immediately.

In Section \ref{sect2} we define a quasilocal energy in the Riemannian setting.
Given compact 
connected $n$-dimensional 
Riemannian manifolds-with-boundary $(N, \partial N)$ and $(M, \partial M)$, suppose that
$f : (N, \partial N) \rightarrow (M, \partial M)$ is a spin map that is an isometry on
$\partial N$.  We think of $M$ as the background
space and $N$ as the manifold whose quasilocal energy we want to define, relative to $M$.  There is
a corresponding Clifford module $E$ on $N$. For simplicity, suppose that $n$ is even dimensional
and that $N$ and $M$ are spin, in
which case $E$ is the twisted
spinor bundle $S_N \otimes f^* S_M^*$.  Let $D^N$ be the corresponding Dirac-type operator on $C^\infty(N; E)$.

On $\partial N$, we have the identifications
\begin{equation}
E \Big|_{\partial N} \cong (S_N \otimes (\partial f)^* S_M^*)  \Big|_{\partial N}
\cong (S_N \otimes S_N^*)  \Big|_{\partial N} \cong \Lambda^* T^* N \Big|_{\partial N} \cong
\Lambda^* T^* \partial N \oplus \Lambda^* T^* \partial N,
\end{equation}
where the last isomorphism is a separation into tangential forms and forms with a normal
component. Let $\pi_+ : \Lambda^* T^* N \Big|_{\partial N} \rightarrow
\Lambda^* T^* \partial N$ be projection onto the tangential component.
There is a natural ``Dirichlet'' boundary condition for sections $\psi$ of $E$ given by 
$\pi_+ \left( \psi \Big|_{\partial N} \right) = 0$; let ${\mathcal D}$
denote the ensuing self-adjoint operator.  For simplicity, in this introduction we mostly consider the
case when $\Ker({\mathcal D}) = 0$. 

 To set up a boundary value problem, there is a
canonical choice of section of $\Lambda^* T^* \partial N$, namely the constant function $1$.
Let $e_n$ be the inward-pointing unit normal on $\partial N$.

\begin{definition}
If $C$ is a
nonempty union of connected components of $\partial N$, let $\psi \in C^\infty(N; E)$ be
the (unique) solution to $D^N \psi = 0$ with $\pi_+ \left( \psi \Big|_{\partial N} \right) = 1_C$, the characteristic
function of $C$.
The quasilocal energy associated to the boundary subset $C$
is \begin{equation} \label{1.2}
{\mathcal E}_C = - \int_{\partial N} \langle \psi,
  \nabla^N_{e_n} \psi \rangle \: \dvol_{\partial N}.
  \end{equation}
\end{definition}

Here are some basic properties of the quasilocal energy. Let $R_N$  and $R_M$ denote the scalar
curvatures of $N$ and $M$, respectively. Define $| \Lambda^2 df |$, the distortion of $f$ on $2$-forms, to be 
the pointwise norm $\sup_{v \wedge w \neq 0} \frac{|df(v) \wedge df(w)|}{|v \wedge w|}$.

We recall that the boundary map $\partial f$ is an isometry on $\partial N$.

\begin{theorem} \label{1.3}
1. If $f$ is an isometric diffeomorphism then ${\mathcal E}_{\partial N} = 0$. \\
2. If $M$ has nonnegative curvature operator and $R_N \ge |\Lambda^2 df| (f^* R_M)$ then
${\mathcal E}_C \ge 0$.
\end{theorem}

Part 1 of Theorem \ref{1.3} is consistent with the interpretation of ${\mathcal E}$ as being a
relative energy between $N$ and $M$.  Part 2 of Theorem \ref{1.3} is an immediate consequence
of a Bochner-type formula.  In particular, if $M$ is flat and $R_N \ge 0$ then ${\mathcal E}_C \ge 0$.

Although (\ref{1.2}) is an integral over the boundary of $N$, the solution $\psi$ to the boundary
value problem depends {\it a priori} on the interior geometry of $N$ and the map $f$.
When $M$ is a domain in $\R^n$, the quasilocal energy only depends on $f$ through its
boundary restriction $\partial f$.

We give a formula for ${\mathcal E}_C$ indicating that in the weak field limit, i.e. when $N$ is
a perturbation of $M$, the quasilocal energy is approximately equal to the Brown-York
energy \cite{Brown-York (1993)}.
We expect, but do not prove, that if one takes an
appropriate exhaustion of an asymptotically flat manifold by compact domains then their
quasilocal energies will approach a normalization constant times the ADM mass.
We compute ${\mathcal E}_{\partial N}$ when $N$ is conformally
related to $M$ and find that it is exactly the Brown-York energy.
We find a similar statement when $N$ is rotationally symmetric.

In Section \ref{sect3} we consider a Lorentzian quasilocal energy.  That is, we have
Lorentzian $(n+1)$-dimensional
manifolds $\overline{N}$ and $\overline{M}$, along with
compact connected spatial hypersurfaces-with-boundary $N \subset \overline{N}$ and 
$M \subset \overline{M}$. Again, we use a comparison map $f : (N, \partial N) \rightarrow
(M, \partial M)$ that is an isometry on $\partial N$ and look at the Dirac-type operator $D^N$ acting on sections of the
twisted spinor bundle $S_N \otimes f^* S_M^*$.  There is a new feature that, in general,
$D^N$ need not be formally self-adjoint on the interior of $N$.  We first look at the
case when $M$ is a totally geodesic hypersurface in $\overline{M}$.  Then there is no
problem with self-adjointness and results from Section \ref{sect2} extend.  The 
curvature condition
in part 2 of Theorem \ref{1.3} gets replaced by
$2 \left( T_{00} - \sqrt{- \sum_{\alpha = 1}^n T_{0\alpha} T^{0 \alpha}}
  \right) \ge
  | \Lambda^2 df | (f^* R_M)$, where $T_{AB} = {R}^{\overline{N}}_{AB} - \frac12
  {R}_{\overline{N}}
g^{\overline{N}}_{AB}$ is the Einstein tensor of $\overline{N}$.

We next look at the case when $\overline{M}$ is the flat Minkowski space
$\R^{n,1}$, but the hypersurface $M$ need not be totally geodesic.  It turns out that $D^N$ is formally
self-adjoint on the interior of $N$ if one uses an appropriate weighted
$L^2$-space.  Regarding boundary conditions, there are two natural choices.
The first one does not give a self-adjoint boundary value problem but nevertheless
one can use it to define a quasilocal energy ${\mathcal E}_C$.  The second one does
give a self-adjoint boundary value problem and gives rise to a quasilocal energy
${\mathcal E}^{sa}_C$.  They have the following properties.

\begin{theorem} \label{1.4}
1. If $f$ is an isometric diffeomorphism then ${\mathcal E}_{\partial N} = 0$. \\
2. If $\overline{N}$ satisfies the dominant energy condition
$T_{00} \ge \sqrt{- \sum_{\alpha = 1}^n T_{0\alpha} T^{0 \alpha}}$ then
${\mathcal E}_C \ge 0$ and ${\mathcal E}^{sa}_C \ge 0$.
\end{theorem}

Comparing ${\mathcal E}_C$ and ${\mathcal E}^{sa}_C$, the first one ${\mathcal E}_C$
has the advantage, from part 1 of Theorem \ref{1.4}, of vanishing when $C = \partial N$ and $f$ is an isometric
diffeomorphism.  On the other hand, ${\mathcal E}^{sa}_C$ has the advantage of coming
from a self-adjoint boundary value problem.

Finally, we extend the results to the case when
$\overline{M}$ is a more general Lorentzian manifold than Minkowski space.  It turns
out that it is enough for $\overline{M}$ to be a product spacetime $\R \times X$ for some
Riemannian manifold $X$.  This can be compared with Chen-Wang-Wang-Yau's construction of a
quasilocal energy when the background space is in a static spacetime
\cite{Chen-Wang-Wang-Yau (2018)}.

Regarding earlier work about a spinorial approach to quasilocal energy, Dougan and Mason
used the complex structure on an oriented surface to make an interesting choice of boundary condition
\cite{Dougan-Mason (1991)}.
Their approach is clearly restricted to $n=3$. Zhang used the pure Dirac operator to
define a quasilocal energy when the background space $M$ is a domain in $\R^3 \subset
\R^{3,1}$ \cite{Zhang (2008)}. His boundary condition came from using constant spinors
on $\R^{3,1}$; compare with Proposition \ref{3.22} of the present paper.  He imposed
geometric restrictions to ensure that the Dirac operator is invertible; compare with
Proposition \ref{3.21} of the present paper.  In \cite{Zhang (2009)} he also allowed $M$ 
to be a domain in a hyperbolic submanifold of $\R^{3,1}$.

I thank the Fields Institute for its hospitality while part of this research was performed, and
the referee for helpful comments.

\section{Riemannian case}  \label{sect2}

This section is about the quasilocal energy in the Riemannian case.  We begin with even dimensional spaces.
Section \ref{subsect2.1} has background information.
Section \ref{subsect2.2} has the definition of the quasilocal energy. In
Section \ref{subsect2.3} we prove its
basic properties. Section \ref{subsect2.4} computes the quasilocal energy for
conformally related manifolds.  Section \ref{subsect2.5} specializes to when the background space
$M$ is a domain in $\R^n$ and also treats rotationally symmetric manifolds $N$. Finally, Section \ref{subsect2.6} covers the odd dimensional case.

\subsection{Background} \label{subsect2.1}

Let $R$ denote scalar
curvature and let $H$ denote mean curvature. With our convention, $\partial B^n$ has
$H = n-1$.

  Let $N$ and $M$ be compact connected
  $n$-dimensional Riemannian manifolds with nonempty boundary.
  Let $f \: : \: (N, \partial N) \rightarrow (M, \partial M)$
  be a smooth spin map,
i.e. $TN \oplus f^* TM$ admits a spin structure. 
Equivalently, $f^* w_2(M) = w_2(N)$. Let
  $\partial f \: : \: \partial N \rightarrow \partial M$
  denote the restriction to the boundary. We assume that for each
  component $Z$ of $\partial N$, the restriction
  $\partial f \Big|_Z$ is an isometric diffeomorphism from $Z$ to $f(Z)$. 

For simplicity, we assume that $N$ and $M$ are spin and that $\partial f$ is a
spin diffeomorphism; the general case is
similar.
We assume first that $n$ is even. Then a spinor representation
$S^\pm$ of $\Spin(n)$
has complex dimension $2^{\frac{n}{2} - 1}$, while the Clifford algebra has a
faithful representation space $S = S^+ \oplus S^-$ of
complex dimension $2^{\frac{n}{2}}$. We let $S^*$ denote the complex
vector space of complex-linear functionals on $S$, i.e. no complex
conjugation involved.
Let $S_N$ denote the spinor bundle on $N$, and similarly for $S_M$.
Put $E = S_N \otimes f^* S^*_M$, a Clifford module on $N$.
(This Clifford module exists in the general case.)
We take the inner product $\langle \cdot, \cdot \rangle$ on $E$
to be $\C$-linear in the
second slot and $\C$-antilinear in the first slot.
As $\End(S_N) \cong \Lambda^*(T^*N)$,
we can identify $E \Big|_{\partial N}$ with
$\Lambda^*(T^* N) \Big|_{\partial N}$. (We require this property in the
general case, which amounts to a choice of spin structure on $E \Big|_{\partial N}$.)

Let $\{e_\alpha\}_{\alpha = 1}^n$ be
a local oriented orthonormal framing on $N$, with dual basis
$\{\tau^\alpha\}_{\alpha = 1}^n$.
Let $\omega^\alpha_{\: \: \beta \gamma}$ be the connection $1$-forms
with respect to $\{e_\alpha\}_{\alpha = 1}^n$. 
Let
$\{\widehat{e}_{\widehat{\alpha}}\}_{\widehat{\alpha} = 1}^n$ of $M$ be
a local oriented orthonormal framing of $M$.
Let
$\widehat{\omega}^{\widehat{\alpha}}_{\: \: \widehat{\beta} \gamma}$ be the pullbacks under $f$ of the
connection $1$-forms with respect to
$\{\widehat{e}_{\widehat{\alpha}}\}_{\widehat{\alpha} = 1}^n$.

To make things more symmetric, it will be convenient
to do local calculations on $S_N \otimes f^* S_M$ rather than
$S_N \otimes f^* S_M^*$.  In the Riemannian setting, $S_M^*$ is unitarily equivalent to $S_M$ and so we don't lose anything this way.
  Let $\{\gamma^\alpha\}_{\alpha = 1}^n$ be Clifford multiplication by
  $\{e_\alpha\}_{\alpha = 1}^n$,
  satisfying $\gamma^\alpha \gamma^\beta + \gamma^\beta \gamma^\alpha =
  2 \delta^{\alpha \beta}$.
  Let
  $\{\widehat{\gamma}^{\widehat{\alpha}} \}_{\widehat{\alpha} = 1}^n$
  be Clifford multiplication by 
  $\{\widehat{e}_{\widehat{\alpha}}\}_{\widehat{\alpha} = 1}^n$. They both have an odd grading, in the sense that
  in calculations we take $\gamma^\alpha$ to anticommute
  with $\widehat{\gamma}^{\widehat{\alpha}}$.
  
The covariant derivative on $E$ has the local form
\begin{equation} \label{2.2}
\nabla^N_\sigma = e_\sigma + \frac{1}{8} \omega_{\alpha \beta \sigma}
[\gamma^\alpha, \gamma^\beta] +
\frac{1}{8} \widehat{\omega}_{\widehat{\alpha} \widehat{\beta} \sigma}
     [\widehat{\gamma}^{\widehat{\alpha}},
       \widehat{\gamma}^{\widehat{\beta}}].
     \end{equation}
The Dirac operator on $C^\infty(N; E)$ is
$D^N = - \sqrt{-1} \sum_{\sigma=1}^n \gamma^\sigma \nabla^N_{\sigma}$.

We will take the orthonormal frame $\{e_\alpha\}$ at a point in
$\partial N$ so that $e_n$ is the inward-pointing unit normal vector
there. Let $\dvol_N$ denote the Riemannian density on $N$, and
similarly for $\dvol_{\partial N}$.
Given $\psi_1, \psi_2 \in C^\infty(N; E)$, we have
\begin{equation} \label{2.3}
\int_N \langle D^N \psi_1, \psi_2 \rangle \dvol_N -
\int_N \langle \psi_1, D^N \psi_2 \rangle \dvol_N  =
- \sqrt{-1} \int_{\partial N} \langle \psi_1, \gamma^n \psi_2 \rangle 
\dvol_{\partial N}.
\end{equation}
The Lichnerowicz formula implies
\begin{equation} \label{2.4}
  (D^N)^2 = (\nabla^N)^* \nabla^N + \frac{R_N}{4} - \frac14 
  [\gamma^\sigma, \gamma^\tau] \left(
  \frac18 \widehat{R}_{\widehat{\alpha} \widehat{\beta} \sigma \tau}
          [\widehat{\gamma}^{\widehat{\alpha}},
  \widehat{\gamma}^{\widehat{\beta}}] \right).
\end{equation}

We now extend some computations in \cite[Proof of Lemma 4.1]{Lott (2000)}.
Suppose that $D^N \psi = 0$. In what follows, summations over Greek letters will go from
$1$ to $n$ and summations over Latin letters will go from $1$ to $n-1$.
Equation (\ref{2.4}) implies that
\begin{align} \label{2.5}
0 = & 
  \int_N |\nabla^N \psi |^2 \: \dvol_N +
\int_{\partial N} \langle \psi, \nabla^N_{e_n} \psi \rangle \dvol_{\partial N}
  + \frac14 \int_N R_N |\psi|^2 \: \dvol_N - \\
  & \frac{1}{32} \int_N
  \widehat{R}_{\widehat{\alpha} \widehat{\beta} \sigma \tau} \langle \psi,
          [\gamma^\sigma, \gamma^\tau]
          [\widehat{\gamma}^{\widehat{\alpha}},
  \widehat{\gamma}^{\widehat{\beta}}] \psi \rangle. \notag
\end{align}
Now $D^N \psi = 0$ implies that on $\partial N$, we have
\begin{align} \label{2.6}
  \nabla^N_{e_n} \psi  = & - \gamma^n \sum_{i = 1}^{n-1}
  \gamma^i \nabla^N_i \psi \\
   = &
  - \gamma^n \sum_{i = 1}^{n-1} \gamma^i
  \left( \nabla^{\partial N}_i \psi + \frac12 \omega_{nj i} \gamma^n
  \gamma^j \psi + \frac12 \widehat{\omega}_{\widehat{n}
    \widehat{j}i}
  \widehat{\gamma}^{\widehat{n}} \widehat{\gamma}^{\widehat{j}}
  \psi \right) \notag \\
  = & D^{\partial N} \psi + \frac{H_{\partial N}}{2} \psi \:  - \:
  \frac12 \gamma^n \gamma^i \widehat{\gamma}^{\widehat{n}}
  \widehat{\gamma}^{\widehat{j}}
  \widehat{A}_{\widehat{j} i} \psi, \notag
\end{align}
where
\begin{equation} \label{2.7}
D^{\partial N} =  - \gamma^n \sum_{i = 1}^{n-1} \gamma^i
   \nabla^{\partial N}_i
\end{equation}
is the Dirac operator on $\partial N$ coupled to
$(\partial f)^* S^*_M$, $\widehat{A}$ is the second fundamental form
of $M$ and $\widehat{A}_{\widehat{j} i} = \widehat{A}(
\widehat{e}_{\widehat{j}},
(\partial f)_*(e_i))$.

\begin{remark} \label{2.8}
The factor of $\frac12$ in (\ref{2.6}) corrects the $\frac14$ that appears in
\cite{Lott (2000),Lott (2021)}.
\end{remark}

From \cite[Section 1.1]{Goette-Semmelmann (2002)}, 
if $M$ has nonnegative curvature operator then
\begin{equation} \label{2.9}
\frac{1}{32}  \widehat{R}_{\widehat{\alpha} \widehat{\beta} \sigma \tau} 
     [\gamma^\sigma, \gamma^\tau]
     [\widehat{\gamma}^{\widehat{\alpha}},
       \widehat{\gamma}^{\widehat{\beta}}]
     \le \frac14 | \Lambda^2 df | (f^* R_M) \Id_E.
  \end{equation}
(The paper \cite{Goette-Semmelmann (2002)} assumes that
$| \Lambda^2 df | \le 1$ but their argument shows the result stated in
(\ref{2.9}); c.f. \cite{Listing (2012)}.)

\begin{proposition} \label{2.10}
If $M$ has nonnegative curvature operator and
$R_N \ge | \Lambda^2(df) | (f^* R_M)$ then for any $\psi \in C^\infty(N; E)$, we have
$- \int_{\partial N} \langle \psi, \nabla^N_{e_n} \psi \rangle \dvol_{\partial N} \ge 0$. If $M$ has nonnegative curvature operator,
$R_N \ge | \Lambda^2(df) | (f^* R_M)$, $\psi$ is nonzero and 
$- \int_{\partial N} \langle \psi, \nabla^N_{e_n} \psi \rangle \dvol_{\partial N} 
= 0$ then  $R_N = | \Lambda^2(df) | (f^* R_M)$.
\end{proposition}
\begin{proof}
This follows from (\ref{2.5}) and (\ref{2.9}).
\end{proof}

\subsection{Definition of the quasilocal energy} \label{subsect2.2}

Let $\epsilon$ be the $\Z_2$-grading operator on $S_N$.  Put
$\gamma^0 = i \epsilon$.  Let $T$
be the involution on $S_N \Big|_{\partial N}$ given by
\begin{equation} \label{2.11}
  T\eta = \gamma^0 \gamma^n \eta.
\end{equation}
Identifying
$f^* S^*_M \Big|_{\partial N}$ with $S^*_N \Big|_{\partial N}$, there is an induced involution $T$ on
\begin{equation}
\Lambda^* T^*N \Big|_{\partial N} \cong \End(S_N) \Big|_{\partial N}
\cong (S_N \otimes S_N^*) \Big|_{\partial N} \cong
S_N \Big|_{\partial N} \otimes f^* S^*_M \Big|_{\partial M}.
\end{equation}
We recall that an element of $\Lambda^* T^*N$ acts on $S_N$ by Clifford multiplication,
realizing the isomorphism $\Lambda^* T^*N \cong \End(S_N)$. Then
for $\omega \in \Lambda^* T^*N \Big|_{\partial N}$ and $\eta \in 
S_N \Big|_{\partial N}$, we have $T(\omega \cdot \eta) = T(\omega) \cdot T(\eta)$.
In terms of the orthonormal frame, $T$ acts on $T^*N \Big|_{\partial N}$ as
$\gamma^0 \gamma^n \widehat{\gamma}^{\widehat{0}} \widehat{\gamma}^{\widehat{n}}$.

Let $\tau^n$ be the
dual covector to $e_n$.

\begin{lemma} \label{2.12}
  Given $\omega \in \Lambda^* T^*N \Big|_{\partial N}$, write
  $\omega = \omega_+ + \tau^n \wedge \omega_-$ for
  $\omega_{\pm} \in \Lambda^* T^* \partial N$. Then
  $T \omega = \omega_+ - \tau^n \wedge \omega_-$.
\end{lemma}
\begin{proof}
  Given an increasing multi-index $I$ with entries between
  $1$ and $n-1$, and
  $\psi \in S_N \Big|_{\partial N}$, we have
  \begin{equation} \label{2.13}
    T(\gamma^I \psi) =  (\gamma^0 \gamma^n) \gamma^I \psi =
      \gamma^I (\gamma^0 \gamma^n) \psi = \gamma^I T(\psi)
  \end{equation}
  and
  \begin{equation} \label{2.14}
    T(\gamma^n \gamma^I \psi) =  (\gamma^0 \gamma^n) \gamma^n \gamma^I \psi =
    - \gamma^n \gamma^I (\gamma^0 \gamma^n) \psi = - \gamma^n \gamma^I T(\psi)
  \end{equation}
  Using the isomorphism $\Lambda^*T^*N \Big|_{\partial N}\cong \End(S_N) \Big|_{\partial N}$,
the lemma follows.
\end{proof}

Let $\pi_{\pm}$ denote orthogonal projection
onto the $\pm 1$-eigenspace of $T$,
acting on $\Lambda^* T^*N \Big|_{\partial N}$, so
$\Image(\pi_+)$ is isomorphic to $\Lambda^* T^* \partial N$.

Consider the operator $D^N$ acting on elements $\psi \in
C^\infty(N; E)$. The boundary condition
$\pi_+ \left( \psi \Big|_{\partial N} \right) = 0$, which is an analog of
Dirichlet boundary conditions, defines an elliptic boundary
condition for $D^N$ \cite[Section 7.5]{Baer-Ballmann (2012)}. One can check
that with the given boundary condition, $D^N$ is formally self-adjoint.
Let ${\mathcal D}$ denote the ensuing self-adjoint operator, densely
defined on $L^2(N; E)$ \cite[Chapter 7]{Baer-Ballmann (2012)}.  It has compact resolvent. The next proposition gives a sufficient
condition for ${\mathcal D}$ to be invertible.

\begin{proposition} \label{2.15}
Suppose that 
\begin{itemize}
\item $M$ has nonnegative curvature operator,
\item $\partial M$ has nonnegative second fundamental form,
\item $R_N \ge | \Lambda^2df | (f^* R_M)$,
\item 
$H_{\partial N} \ge (\partial f)^* H_{\partial M}$, and
\item $R_N > | \Lambda^2df | (f^* R_M)$
 somewhere or $H_{\partial N} > (\partial f)^* H_{\partial M}$
somewhere.
\end{itemize}
Then $\Ker({\mathcal D}) = 0$. 
\end{proposition}
\begin{proof}
Suppose that $\psi \in \Ker({\mathcal D})$.  
From (\ref{2.5}) and (\ref{2.6}),
\begin{align} \label{2.16} 
0 =  & \int_N |\nabla^N \psi |^2 \: \dvol_N
  + \frac14 \int_N R_N |\psi|^2 \: \dvol_N   
- \frac{1}{32} \int_N
  \widehat{R}_{\widehat{\alpha} \widehat{\beta} \sigma \tau} \langle \psi,
          [\gamma^\sigma, \gamma^\tau]
          [\widehat{\gamma}^{\widehat{\alpha}},
  \widehat{\gamma}^{\widehat{\beta}}] \psi \rangle
 + \\
& \int_{\partial N} \langle \psi, D^{\partial N}
  \psi \rangle \: \dvol_{\partial N} + 
\frac12
  \int_{\partial N} H_{\partial N} |\psi|^2 \: \dvol_{\partial N} - 
 \frac12
  \int_{\partial N} \langle \psi, \gamma^n \gamma^{i}
  \widehat{\gamma}^n \widehat{\gamma}^{\widehat{j}}
  \widehat{A}_{\widehat{j}i} \psi
  \rangle \: \dvol_{\partial N}. \notag
\end{align}
As $T$ anticommutes with $D^{\partial N}$, we have
\begin{equation} \label{2.17}
\langle \psi, D^{\partial N}
  \psi \rangle = - \langle \psi, D^{\partial N}
 T \psi \rangle = 
\langle \psi, T D^{\partial N}
  \psi \rangle =
\langle T \psi, D^{\partial N}
  \psi \rangle =
- \langle \psi, D^{\partial N}
  \psi \rangle,
\end{equation}
so
$\langle \psi, D^{\partial N}
  \psi \rangle= 0$.  From \cite[Lemma 2.1]{Lott (2021)},
$\langle \psi, \gamma^n \gamma^{i}
  \widehat{\gamma}^n \widehat{\gamma}^{\widehat{j}}
  \widehat{A}_{\widehat{j}i} \psi
  \rangle \le (\partial f)^* H_{\partial M} |\psi|^2$.
Then (\ref{2.9}) and (\ref{2.16}) imply that $\nabla^N \psi = 0$. Since $\nabla^N$ is unitary,
it follows that $|\psi|^2$ is constant on the connected manifold $N$.  Then
(\ref{2.16}) implies that $|\psi|^2 = 0$.
\end{proof}

\begin{remark} \label{2.18}
    If the Euler characteristic of $M$ is nonzero and the degree of $f$ is nonzero then $\Ker({\mathcal D}) \neq 0$
    \cite[Section 2.2]{Lott (2021)}.
\end{remark}

We now discuss the boundary value problem.
Given $\sigma \in \Omega^*(\partial N)$ with $\pi_+(\sigma) = \sigma$, 
we wish to find some
$\psi \in C^\infty(N; E)$ so that
  $D^N \psi = 0$ and $\pi_+ \left( \psi \Big|_{\partial N} \right)
  = \sigma$. Suppose that we can do this. If
$\eta \in \Ker({\mathcal D})$ then from (\ref{2.3}), we see that
$\int_{\partial N} \langle \sigma, \gamma^n \eta \rangle \: \dvol_{\partial N} = 0$.
(Note that $\eta \Big|_{\partial N} \in \Image(\pi_-)$, so 
$\gamma^n \eta \Big|_{\partial N} \in \Image(\pi_+)$.)

Conversely, given $\sigma \in \Omega^*(\partial N)$ with $\pi_+(\sigma) = \sigma$, if 
$\int_{\partial N} \langle \sigma, \gamma^n \eta \rangle \: \dvol_{\partial N} = 0$
for all $\eta \in \Ker({\mathcal D})$ then there is some
$\psi \in C^\infty(N; E)$ so that
  $D^N \psi = 0$ and $\pi_+ \left( \psi \Big|_{\partial N} \right)
  = \sigma$ \cite[Lemma 14]{Farinelli-Schwarz (1998)}. The section
$\psi$ is unique up to addition by elements of $\Ker({\mathcal D})$.

Let $C$ be a nonempty union of connected components of $\partial N$.
To define the quasilocal energy, we would like to take $\sigma$ to be the characteristic function $1_C$. 
To be sure that $\psi$ exists, we may
have to slightly modify this choice.  Let 
$P_{\gamma^n \Ker({\mathcal D}) \Big|_{\partial N}}$ be 
orthogonal projection
from $\Omega^*(\partial N)$ to $\gamma^n \Ker({\mathcal D}) \Big|_{\partial N}$.
We put $\overline{1}_C = 1_C - 
P_{\gamma^n \Ker({\mathcal D}) \Big|_{\partial N}} 1_C$.
We then take $\sigma = \overline{1}_C$
and construct $\psi$ accordingly.

\begin{definition} \label{2.19}
Put 
\begin{equation} \label{2.20}
{\mathcal C}_C = \left\{ \psi \in C^\infty(N; E) \: : \: D^N \psi = 0,
\pi_+ \left( \psi \Big|_{\partial N} \right)
  = \overline{1}_C \right\}.
\end{equation}
  The quasilocal energy associated to the boundary subset $C$ is
\begin{equation} \label{2.21}
{\mathcal E}_C = \inf_{\psi \in {\mathcal C}_C} \left( - \int_{\partial N} \langle \psi,
  \nabla^N_{e_n} \psi \rangle \: \dvol_{\partial N} \right).
\end{equation}
\end{definition}

Note that ${\mathcal E}_C$ could be $-\infty$.

\begin{lemma} \label{2.22}
If $\Ker({\mathcal D}) = 0$ then ${\mathcal E}_C > - \infty$.  In general,
if ${\mathcal E}_C > - \infty$ then it is realized by some 
$\psi_{min} \in {\mathcal C}_C$.
\end{lemma}
\begin{proof}
We know that ${\mathcal C}_C$ is a finite dimensional affine space modelled on $\Ker({\mathcal D})$.  If $\Ker({\mathcal D}) = 0$ then ${\mathcal C}_C$ has a single element and 
${\mathcal E}_C > - \infty$. 
If $\Ker({\mathcal D}) \neq 0$ then as
$- \int_{\partial N} \langle \psi,
  \nabla^N_{e_n} \psi \rangle \: \dvol_{\partial N}$ is a quadratic function on
  ${\mathcal C}_C$, the lemma follows.
\end{proof}

In some statements that follow we may implicitly assume that ${\mathcal E}_C > - \infty$. We will see relevant examples where this is the case
in Proposition \ref{2.26}, Proposition \ref{2.27}, Lemma \ref{2.42} and 
Proposition \ref{rot}.

\subsection{Properties of the quasilocal energy} \label{subsect2.3}

\begin{proposition} \label{2.23}
If $f : N \rightarrow M$ is an isometric diffeomorphism then
${\mathcal E}_{\partial N} = 0$.
\end{proposition}
\begin{proof}
  If $f$ is an isometric diffeomorphism then we can take
  $E \cong \Lambda^* T^* N$ and identify $D^N$ with $d+d^*$. 
  Letting $*$ denote the Hodge duality operator, if $\psi \in C^\infty(N; E)$ then
\begin{align} \label{2.24}
\int_N \langle \left( D^N \right)^2 \psi, \psi \rangle \: \dvol_N & = \int_N \langle (dd^* + d^* d)\psi, \psi \rangle \: \dvol_N \\
& =
\int_N \left( dd^*\psi \wedge *\psi + d^*d\psi \wedge *\psi \right) \notag \\
& =
\int_N \left( dd^*\psi \wedge *\psi + *d^*d\psi \wedge *(*\psi) \right)  \notag \\
& =
\int_N \left( dd^*\psi \wedge *\psi + dd^**\psi \wedge *(*\psi) \right) 
\notag \\
& = \int_N \left( |d^* \psi|^2 + |d\psi|^2 \right) \: \dvol_N +
\int_{\partial N} \left( d^*\psi \wedge *\psi + d^**\psi \wedge *(*\psi) \right). \notag
\end{align}
Comparing with equation (\ref{2.5}) shows
\begin{equation} \label{2.25}
- \int_{\partial N} \langle \psi,
  \nabla^N_{e_n} \psi \rangle \: \dvol_{\partial N} = - \int_{\partial N} \left( d^* \psi \wedge *\psi + d^* * \psi 
\wedge *(* \psi) \right),
\end{equation}
  
  There is a
  solution $\psi$ of $D^N \psi = 0$
  given by $\psi = 1$, with $\pi_+(1) = 1$.
It follows that $1$ is orthogonal to
$\gamma^n \Ker({\mathcal D}) \Big|_{\partial N}$, so 
$\overline{1}_{\partial N} = 1$.

If $i \: : \: \partial N \rightarrow N$ is the boundary inclusion then
an element of $\Ker({\mathcal D})$ is a differential 
form $\rho \in \Omega^*(N)$
satisfying $(d + d^*) \rho = 0$ and the relative (Dirichlet) boundary condition
$i^* \rho = 0$. Applying $i^*$, it follows that $i^* (d^* \rho) = 0$.
Squaring $d + d^*$, it follows that $(dd^* + d^* d) \rho = 0$. Hence
$\rho$ is a 
harmonic form $\alpha$ on $N$ satisfying relative (Dirichlet) boundary
  conditions. In particular, $d \rho = d^* \rho = 0$.

Hence elements $\psi$ 
of ${\mathcal C}_{\partial N}$ are of the form $1+\rho$ for such $\rho$.
As
 $d \psi = d^* \psi = 0$,
equation (\ref{2.25}) gives $- \int_{\partial N} \langle \psi,
  \nabla^N_{e_n} \psi \rangle \: \dvol_{\partial N} =0$, which proves the lemma.
  \end{proof}

We now give the basic positivity property of the quasilocal energy in the
physically relevant case.

\begin{proposition} \label{2.26}
  Suppose that $M$ has nonnegative curvature operator.
  If $R_N \ge | \Lambda^2 df | (f^* R_M)$ then ${\mathcal E}_C \ge 0$.  If
  $R_N \ge | \Lambda^2 df | (f^* R_M)$ and ${\mathcal E}_C = 0$ then
  $R_N = | \Lambda^2 df | (f^* R_M)$.
\end{proposition}
\begin{proof}
This follows from (\ref{2.5}) and (\ref{2.9}).
  \end{proof}

\begin{proposition} \label{2.27}
  If $M$ is flat and $R_N \ge 0$ then ${\mathcal E}_C \ge 0$. If
  $M$ is flat, $R_N \ge 0$ and ${\mathcal E}_C = 0$ then $N$ is Ricci flat.
   \end{proposition}
\begin{proof}
The first statement follows from Proposition \ref{2.26}.  The second 
statement follows as in \cite[Proof of Proposition 2.3]{Lott (2021)}.
\end{proof}

We now write ${\mathcal E}_C$ more explicitly as a boundary integral. From (\ref{2.6}), we have
\begin{align} \label{2.28}
  {\mathcal E}_C =  & - \int_{\partial N} \langle \psi_{min}, D^{\partial N}
  \psi_{min} \rangle \: \dvol_{\partial N} - \frac12
  \int_{\partial N} H_{\partial N} |\psi_{min}|^2 \: \dvol_{\partial N} + \\
  & \frac12
  \int_{\partial N} \langle \psi_{min}, \gamma^n \gamma^{i}
  \widehat{\gamma}^n \widehat{\gamma}^{\widehat{j}}
  \widehat{A}_{\widehat{j}i} \psi_{min}
  \rangle \: \dvol_{\partial N}. \notag
  \end{align}
We write $\psi_{min} \Big|_{\partial N} = 
\overline{1}_C + \tau^n \wedge \phi$ for some
$\phi \in \Omega^*(\partial N)$. Let $E^i$ denote exterior multiplication
by $\tau^i$ and $I^i$ denote interior multiplication by
$e_i$, when 
acting on $\Lambda^* T^* \partial N +
\tau^n \wedge \Lambda^* T^* \partial N$.

\begin{proposition} \label{2.29} 
If $\overline{1}_C = 1_C$ then
\begin{align} \label{2.30}
  {\mathcal E}_C = & - \frac12
  \int_{C} H_{\partial N} \: \dvol_{\partial N} + \frac12
  \int_{C} (\partial f)^*H_{\partial M} \: \dvol_{\partial N} - \\
& 
\frac12
\int_{\partial N} H_{\partial N} |\phi|^2 \: \dvol_{\partial N} + 
\int_{\partial N}
\frac12 \widehat{A}_{ji}
\langle \phi, (E^i - I^i) (E^j + I^j)
\phi \rangle \: \dvol_{\partial N}.
\notag
  \end{align}
\end{proposition}
\begin{proof}
On $\partial N$, we can represent $\gamma^i$ by
$\sqrt{-1}(E^{i} - I^{i})$, $\gamma^n$ by
$\sqrt{-1} (E^n - I^n)$, $\widehat{\gamma}^i$ by $E^i + I^i$ and
$\widehat{\gamma}^n$ by
$E^n + I^n$.

\begin{lemma} \label{2.31}
We have
  \begin{equation} \label{2.32}
\int_{\partial N} \langle \psi_{min}, D^{\partial N}
\psi_{min} \rangle \: \dvol_{\partial N} = 0.
  \end{equation}
\end{lemma}
\begin{proof}
  As $D^{\partial N}$ anticommutes with $T$, it follows that
\begin{equation} \label{2.33}
  \int_{\partial N} \langle 1_C, D^{\partial N}
  1_C \rangle \: \dvol_{\partial N} =
  \int_{\partial N} \langle \tau^n \wedge \phi, D^{\partial N}
(\tau^n \wedge \phi) \rangle \: \dvol_{\partial N} = 0.
\end{equation}
Next, since $\nabla^{\partial N}$ restricts to the Riemannian connection on
$\Lambda^* T^* \partial N$, we know that 
$\nabla^{\partial N} 1_C = 0$, so
\begin{equation} \label{2.34}
    \int_{\partial N} \langle \tau^n \wedge \phi, D^{\partial N}
    1_C \rangle \: \dvol_{\partial N} = 0
  \end{equation}
    and
\begin{equation} \label{2.35}
    \int_{\partial N} \langle 1_C, D^{\partial N}
  \tau^n \wedge \phi \rangle \: \dvol_{\partial N} =
  \int_{\partial N} \langle D^{\partial N} 1_C, 
  \tau^n \wedge \phi \rangle \: \dvol_{\partial N} = 0
\end{equation}
 The lemma follows.
  \end{proof}
  
Also
\begin{equation} \label{2.36}
  \int_{\partial N} H_{\partial N} |\psi_{min}|^2 \: \dvol_{\partial N} =
  \int_{C} H_{\partial N} \: \dvol_{\partial N}  +
  \int_{\partial N} H_{\partial N} |\phi|^2 \: \dvol_{\partial N}.
\end{equation}

Next, one can check that 
\begin{equation} \label{2.37}
  \langle \psi_{min}, \gamma^n \gamma^{i}
  \widehat{\gamma}^n \widehat{\gamma}^j \psi_{min} \rangle =
  \delta^{ij} 1_C + \langle \phi, (E^i - I^i) (E^j + I^j) \phi \rangle.
\end{equation}
Hence
\begin{equation} \label{2.38}
  \langle \psi_{min}, \gamma^n \gamma^{i}
  \widehat{\gamma}^n \widehat{\gamma}^j \widehat{A}_{ji} \psi_{min} \rangle =
  (\partial f)^* H_{\partial M} 1_C + \widehat{A}_{ji}
  \langle \phi, (E^i - I^i) (E^j + I^j) \phi \rangle.
\end{equation}

This proves the proposition.
\end{proof}

The first two terms on the right-hand side of (\ref{2.30}) give the Brown-York energy. (We have taken a particular normalization of the Brown-York energy.)
In the case when $N$ is a small perturbation of $M$, the function $\phi$ will also be small.
Hence Proposition \ref{2.29} shows that in this weak field limit, to leading order
the quasilocal energy
${\mathcal E}_C$ is the Brown-York energy.

\begin{remark}
With reference to Proposition \ref{2.29}, in the asymptotically flat case the geometry on most of a large domain $N$ will be a small perturbation of the Euclidean metric.  Furthermore, it's known that the Brown-York energy of large spheres approaches the ADM mass
(\cite{Fan-Shi-Tam (2009)} and references therein). This makes it plausible that under an appropriate exhaustion of an asymptotically flat manifold, the quasilocal mass will approach the ADM mass. That this is true in the rotationally symmetric case follows from Proposition \ref{rot}.
\end{remark}

\subsection{Conformal deformations} \label{subsect2.4}

In this section we use the conformal covariance of the Dirac operator to 
say something about the quasilocal energy; c.f. \cite[Section 4]{Zhang (2008)}.

Put $(N^\prime, g^\prime) = (N, e^{2 \phi} g)$, where 
$\phi \Big|_{\partial N}$ is locally constant and
$\phi \Big|_{C} = 0$.
Let $f : N \rightarrow M$ be as before.  Suppose that $\Ker({\mathcal D}) = 0$.
Then $\overline{1}_C = 1_C$.
The pure Dirac operator $Dirac^{N^\prime}$ is related to the pure Dirac operator
$Dirac^N$ by $Dirac^{N^\prime} = e^{- \frac{n+1}{2}\phi} Dirac^N
e^{\frac{n-1}{2}\phi}$ \cite[Section II]{Lott (1986)}.   
Thinking of $S_N \otimes f^* S_M^*$ as $\Hom(f^* S_M, S_N)$, the same argument gives $D^{N^\prime} = e^{- \frac{n+1}{2}\phi} D^N
e^{\frac{n-1}{2}\phi}$. 

Let $\psi$ be a 
minimizer for the quasilocal energy of $N$ relative to $M$, with
$\pi_+ \left( \psi \Big|_{\partial N} \right) = 1_C$. Putting $\psi^\prime =
e^{- \frac{n-1}{2} \phi} \psi$, it satisfies $D^{N^\prime} \psi^\prime = 0$
with $\pi_+ \left( \psi^\prime \Big|_{\partial N^\prime} \right) = 1_C$. As 
$\Ker({\mathcal D}^\prime) = 0$, it follows that 
\begin{equation} \label{2.39}
{\mathcal E}^\prime_C = - \int_{\partial N^\prime} \langle \psi^\prime,
\nabla^{N^\prime}_{e_n^\prime} \psi^\prime \rangle \dvol_{g^\prime}.
\end{equation}
From \cite[Section II]{Lott (1986)},
\begin{equation} \label{2.40}
\nabla^{N^\prime}_{e_n^\prime} = e^{- \phi} \left( \nabla^N_{e_n} +
\frac14 (e_\alpha \phi) [\gamma^n, \gamma^\alpha] \right).
\end{equation}
Applying this to (\ref{2.39}), and using the fact that $\phi \Big|_{\partial N}$
is locally constant,
one finds
\begin{equation} \label{2.41}
{\mathcal E}^\prime_C - {\mathcal E}_C = \frac{n-1}{2} \int_{\partial N} (e_n \phi) |\psi|^2 \:
\dvol_{\partial N}.
\end{equation}

We now specialize to the case when $N = M$, $f = \Id$ and $C = \partial N$. We no longer have $\Ker({\mathcal D}) = 0$.

\begin{lemma} \label{2.42}
If $e_n \phi \ge 0$ then
${\mathcal E}^\prime_{\partial N^\prime} = \frac{n-1}{2} \int_{\partial N} (e_n \phi) \:
\dvol_{\partial N}$. If $e_n \phi < 0$ then ${\mathcal E}^\prime_{\partial N^\prime} = - \infty$.
\end{lemma}
\begin{proof}
Note that $\phi \Big|_{\partial N}=0$.
Given $\psi^\prime \in C^\infty(N^\prime; E^\prime)$
satisfying $D^{N^\prime} \psi^\prime = 0$ and $\pi_+ \left( \psi^\prime \Big|_{\partial N} \right) = \overline{1}^\prime_{\partial N}$, put
$\psi = e^{\frac{n-1}{2} \phi} \psi^\prime$. Then
$D^{N} \psi = 0$ and $\pi_+ \left( \psi \Big|_{\partial N} \right) = \overline{1}_{\partial N}$.  The proof of Proposition \ref{2.23}
implies that $\psi = 1 + \rho$ for some harmonic form $\rho$ satisfying relative boundary conditions.  As
in the proof of Proposition \ref{2.23}, we know that
$ \int_{\partial N} \langle \psi,
\nabla^{N}_{e_n} \psi \rangle \dvol_{g} = 0$.
Since
\begin{equation} \label{2.43}
- \int_{\partial N^\prime} \langle \psi^\prime,
\nabla^{N^\prime}_{e_n^\prime} \psi^\prime \rangle \dvol_{g^\prime} +
\int_{\partial N} \langle \psi,
\nabla^{N}_{e_n} \psi \rangle \dvol_{g}
= \frac{n-1}{2} \int_{\partial N} (e_n \phi) |\psi|^2 \:
\dvol_{\partial N},
\end{equation}
if $e_n \phi \ge 0$ then we minimize $\frac{n-1}{2} \int_{\partial N} (e_n \phi) |\psi|^2 \:
\dvol_{\partial N}$ by taking $\psi = 1$.  As $\dvol_N \in \Ker({\mathcal D})$, if 
$e_n \phi < 0$ then we can make $\frac{n-1}{2} \int_{\partial N} (e_n \phi) |\psi|^2 \:
\dvol_{\partial N}$ arbitrarily negative by taking $\psi = 1 + s \dvol_N$ with $s$ large.
\end{proof}

One can check that the mean curvatures of $\partial N^\prime$ and
$\partial N$ are related by
\begin{equation} \label{2.44}
H_{\partial N^\prime} = H_{\partial N} - (n-1) e_n \phi.
\end{equation}
Hence if $e_n \phi \ge 0$ then
\begin{equation} \label{2.45}
    {\mathcal E}^\prime_{\partial N^\prime} = 
    \frac12 \int_{\partial N} \left( H_{\partial N} - 
    H_{\partial N^\prime} \right) \: \dvol_{\partial N} =
    \frac12 \int_{\partial N^\prime} \left( H_{\partial N} - 
    H_{\partial N^\prime} \right) \: \dvol_{\partial N^\prime},
    \end{equation}
    showing that the quasilocal energy equals the
    Brown-York energy.

In this conformal setting, we can also express the quasilocal energy as an interior integral.
If $n > 2$ then
\begin{align} \label{2.46}
\frac{n-1}{2} \int_{\partial N} (e_n \phi) \:
\dvol_{\partial N} & = \frac{n-1}{n-2} \int_{\partial N} \left( e_n 
e^{\frac{(n-2) \phi}{2}} \right) \: \dvol_{\partial N} \\
& =
- \frac{n-1}{n-2} \int_N \triangle
e^{\frac{(n-2) \phi}{2}} \: \dvol_N \notag \\
&  =
\frac14 \int_N e^{(\frac{n}{2}+1) \phi} \left( R_{g^\prime} - 
e^{-2 \phi} R_g \right) \dvol_N. \notag
\end{align}
Hence if $e_n \phi \ge 0$ then
\begin{equation} \label{2.47}
 {\mathcal E}^\prime_{\partial N^\prime}= 
\frac14 \int_N e^{(\frac{n}{2}+1) \phi} \left( R_{g^\prime} - 
e^{-2 \phi} R_g \right) \dvol_N 
\end{equation}
One can check that this is also true when $n=2$.

As $| \Lambda^2 (d\Id) | = e^{-2\phi}$, for $\Id \: : (N^\prime, g^\prime)
\rightarrow (N, g)$, equation (\ref{2.47}) is consistent with Proposition \ref{2.26}.

As a special case, let $(M, g_{Eucl})$ be a compact connected
codimension-zero submanifold of
$\R^2$ with nonempty boundary. If $\phi \in C^\infty(M)$ vanishes on
$\partial M$ and satisfies $e_n \phi \ge 0$ there, put $g^\prime = e^{2 \phi} g_{Eucl}$. 
Then the quasilocal energy of $\left( M, g^\prime \right)$,
relative to $\left( M, g_{Eucl} \right)$, is 
\begin{equation} \label{2.48}
{\mathcal E}_{\partial M} = \frac14 \int_M e^{2 \phi} R_{g^\prime} \: \dvol_{g_{Eucl}} =
\frac14 \int_M R_{g^\prime} \: \dvol_{g^\prime}.
\end{equation}

\subsection{Background space in $\R^n$} \label{subsect2.5}

In this section we make the quasilocal energy more explicit when the
background space $M$ is a domain in $\R^n$. We also treat the case of
rotationally symmetric $N$.

From Corollary \ref{2.27}, if $M \subset \R^n$ and $R_N \ge 0$ then
${\mathcal E}_C \ge 0$.

\begin{proposition} \label{2.49}
If $M \subset \R^n$ then two choices of $f : N \rightarrow M$ with the same boundary restriction $\partial f$ will give the same value for ${\mathcal E}_C$.
\end{proposition}
\begin{proof}
This is because the construction of ${\mathcal E}_C$ only involves $f$ through
the boundary condition on $\partial f$ and the pullback of 
the connection on $S_M^*$. Since $S_M^*$ is a trivial bundle with trivial
connection, the pullbacks of $S_M^*$ under two choices of $f$ will be equivalent. 
\end{proof}

Of course, ${\mathcal E}_C$ still depends on the intrinsic geometry of $\partial N$, the extrinsic geometry of $\partial N$ (resp. $\partial M$) in $N$ (resp. $M$), and {\it a priori} the interior geometry of
$N$.

\begin{proposition} \label{2.51}
Whenever $\Ker({\mathcal D}) = 0$, the quasilocal energy can be
described as follows. 
Within the spinor module  associated to $\R^n$,
let $\{\epsilon_a\}_{a=1}^{2^{\frac{n}{2}-1}}$ be an orthonormal 
basis for $\Ker(\gamma^0 \gamma^n - I)$ and let 
$\{\epsilon^\prime_a\}_{a=1}^{2^{\frac{n}{2}-1}}$ be an orthonormal 
basis for $\Ker(\gamma^0 \gamma^n + I)$.
Extending $\epsilon_a$ and $\epsilon^\prime_a$ to constant-valued sections of
$S_M$, let
$\psi_a \in C^\infty(N; S_N)$ be a harmonic spinor on $N$ with boundary 
value in $(\partial f)^* \epsilon_a \cdot 1_C + \Ker(\gamma^0 \gamma^n + I)$ and let
$\psi^\prime_a \in C^\infty(N; S_N)$ be a harmonic spinor on $N$ with boundary 
value in $(\partial f)^* \epsilon^\prime_a \cdot 1_C + \Ker(\gamma^0 \gamma^n - I)$
Then
\begin{equation} \label{2.52}
{\mathcal E}_C = - 2^{- n/2} \sum_{a=1}^{2^{\frac{n}{2} - 1}} \int_{\partial N}
\left( \langle \psi_a, \nabla^N_{e_n} \psi_a \rangle  +
\langle \psi^\prime_a, \nabla^N_{e_n} \psi^\prime_a \rangle
\right) \dvol_{\partial N}.
\end{equation}
In particular, this is true under the hypotheses of Proposition \ref{2.15}, 
\end{proposition}
\begin{proof}
Let $\epsilon_a^* \in S_M^*$ denote inner product with $\epsilon_a$, and
similarly for $\epsilon_a^{\prime, *}$.
Under the isomorphism $S_M \otimes S_M^* \cong \Lambda^*T^*M$, we claim that
$\sum_a \left(
\epsilon_a \otimes \epsilon_a^* + \epsilon^\prime_a \otimes 
\epsilon_a^{\prime,*} \right)$ corresponds to
$2^{n/2} \in \Lambda^*T^*M$. This is because if $I$ is a nontrivial increasing
multi-index with entries from $\{1, \ldots, n\}$ then
$\Tr \left( \gamma^I \frac{\Id - \gamma^0 \gamma^n}{2} \right) +
\Tr \left( \gamma^I \frac{\Id + \gamma^0 \gamma^n}{2} \right) = 0$, 
while $\Tr \left( \frac{\Id - \gamma^0 \gamma^n}{2} \right) +
\Tr \left( \frac{\Id + \gamma^0 \gamma^n}{2} \right) = 2^{n/2}$.

On the other hand, if $\eta_a \in \Ker(\gamma^0 \gamma^n + I)$ then
$T(\eta_a \otimes \epsilon_a^*) = - \eta_a \otimes \epsilon_a^*$, and
similarly for an element of the form 
$\eta_a^\prime \otimes \epsilon_a^{\prime,*}$ with
$\eta_a^\prime \in \Ker(\gamma^0 \gamma^n + I)$.
Thus 
\begin{equation} \label{2.53}
\Psi = 2^{-n/2} 
\sum_a \left( \psi_a \otimes f^* \epsilon_a^*
+ \psi^\prime_a \otimes f^* \epsilon_a^{\prime,*} \right)
\end{equation}
satisfies
$D^N \Psi = 0$ and $\pi_+ \left( \Psi \Big|_{\partial N} \right) = 1_C$. 
As $\Ker({\mathcal D}) = 0$, it is the unique such solution. Hence
\begin{align} \label{2.54}
{\mathcal E}_C = & - \int_{\partial N} \langle \Psi, \nabla^N_{e_n} \Psi \rangle \:
\dvol_{\partial N} \\
 = & - 2^{- n/2} \sum_{a=1}^{2^{\frac{n}{2}-1}} \left( \int_{\partial N}
\langle \psi_a, \nabla^N_{e_n} \psi_a \rangle \dvol_{\partial N} + 
\langle \psi^\prime_a, \nabla^N_{e_n} \psi^\prime_a \rangle \dvol_{\partial N} \right).
\notag
\end{align}
This proves the proposition.
\end{proof}

We now consider the case when $N$ is diffeomorphic to a disk, with a rotationally symmetric metric $g_N$. Then $(N, g_N)$ is conformally 
equivalent to a disk in $\R^n$ and by rescaling the disk, 
we can assume that the conformal
equivalence $f$ is an isometry on the boundary.  
   \begin{proposition} \label{rot}
In the rotationally symmetric case, if $H_{\partial N} > 0$ and
the sectional curvature
of $\partial N$ is $\frac{1}{k^2}$ then
\begin{equation}
    {\mathcal E}_{\partial N} = 
    \frac12 \int_{\partial N} \left( \frac{n-1}{k} - H_{\partial N} 
     \right) \: \dvol_{\partial N}.
    \end{equation}
    \end{proposition}
    \begin{proof}
        This follows from (\ref{2.45})
    \end{proof}

\subsection{Odd dimensional spaces} \label{subsect2.6}

We now assume that $n$ is odd. 
The Clifford algebra has a faithful
representation $S$ of complex dimension $2^{\frac{n+1}{2}}$, which breaks
up into two isomorphic spinor representations $S^+ \oplus S^-$ of
$\Spin(n)$.
Explicitly, we can write Clifford generators
$\{\gamma^\alpha\}_{\alpha = 1}^n$
on $S$ as $\gamma^\alpha = 
\begin{pmatrix}
0 & \sigma^\alpha \\
\sigma^\alpha & 0
\end{pmatrix}$, where
$\{\sigma^\alpha\}_{\alpha = 1}^n$ satisfy the Clifford relations on
$\C^{2^{\frac{n-1}{2}}}$. We again let $\epsilon$ be the $\Z_2$-grading
operator on $S$.

With this definition of the spinors $S$,
put $E = S_N \otimes f^* S_M^*$.
We can identify $E \Big|_{\partial N} \cong
\End \left( S_N \Big|_{\partial N} \right)$ with
$\Lambda^*(T^* N) \Big|_{\partial N} \oplus
\Lambda^*(T^* N) \Big|_{\partial N}$. To realize this identification explicitly,
put
$\gamma^0 =
\begin{pmatrix}
  0 & -1 \\
  1 & 0
\end{pmatrix}$.
Let $I$ be an increasing multi-index with
entries between $1$ and $n$.
Let $\tau^0$ denote a new odd variable.  Then an element
$\omega_I \tau^I + \tau^0 \wedge \omega^\prime_I \tau^I$ acts on  
$S_N \Big|_{\partial N}$ by sending $\psi$ to
$(\omega_I \gamma^I + \gamma^0 \omega^\prime_I \gamma^I) \psi$.

Define $T$ as in (\ref{2.11}).
For $\omega_1, \omega_2, \omega^\prime_1, \omega^\prime_2 \in
\Lambda^* T^* \partial N$, we have
\begin{align} \label{2.55}
& \left( (\omega_1 + \tau^n \wedge\omega_2) +   
\tau^0 \wedge (\omega^\prime_1 + \tau^n \wedge \omega^\prime_2) \right)
\gamma^0 \gamma^n \psi = \\
& \gamma^0 \gamma^n \left( (\omega_1 - \tau^n \wedge\omega_2) +  
\tau^0 \wedge (- \omega^\prime_1 + \tau^n \wedge \omega^\prime_2) \right) \psi. \notag
\end{align}
Hence the induced action of $T$ on 
$E \Big|_{\partial N} \cong \End(S_N) \Big|_{\partial N} $
is
\begin{equation} \label{2.56}
  T \left( \omega_1 + \tau^n \wedge\omega_2, 
\omega^\prime_1 + \tau^n \wedge \omega^\prime_2 \right) = 
(\omega_1 - \tau^n \wedge\omega_2,
- \omega^\prime_1 + \tau^n \wedge \omega^\prime_2).
  \end{equation}

We consider solutions to the Dirac equation $D^N \psi = 0$ on sections
$\psi \in C^\infty(N; 
E)$ with the boundary condition that if
$\psi \Big|_{\partial N} = \left( \omega_1 + \tau^n \wedge\omega_2, 
\omega^\prime_1 + \tau^n \wedge \omega^\prime_2 \right)
$ then
$\omega_1 = \overline{1}_C$ and $\omega_2 = \omega^\prime_1 = \omega^\prime_2 = 0$, 
where the overline denotes an orthogonal 
projection as before. Then the results of the previous
sections have straightforward extensions.

\section{Lorentzian case} \label{sect3}

This section is about the extension of Section \ref{sect2} to
hypersurfaces in
Lorentzian manifolds.
Section \ref{subsect3.0} has background material. Section \ref{subsect3.1} discusses
the case when the background space $M$ is a totally geodesic hypersurface in a
Lorentzian manifold $\overline{M}$. Section \ref{subsect3.2} deals with the case when
$M$ is a compact spatial hypersurface-with-boundary in $\R^{n,1}$.  Finally, in
Section \ref{subsect3.3} we describe how the results of Section \ref{subsect3.2} 
extend to when $M$ is a compact spatial hypersurface-with-boundary in a
product spacetime $\R \times X$.

\subsection{Background information} \label{subsect3.0}

Let $\overline{N}$ and $\overline{M}$ be
$(n+1)$-dimensional Lorentzian manifolds with
signature $(-1, 1, \ldots, 1)$. Let $N$ and $M$ be compact connected
$n$-dimensional spacelike
submanifolds of $\overline{N}$ and $\overline{M}$,
respectively, with nonempty boundary. 
Let $f : N \rightarrow M$ be a smooth spin map so that for each
connected component $Z$ of $\partial N$, the map $\partial f$ restricts to
an isometric diffeomorphism from $Z$ to $(\partial f)(Z)$.
By shrinking
$\overline{N}$ and $\overline{M}$ to suitable neighborhoods of $N$ and $M$,
respectively, we can assume that $f$ is the restriction of a spin map
$\overline{f} : \overline{N} \rightarrow \overline{M}$.  For simplicity,
we will assume that $\overline{N}$ and $\overline{M}$ are spin and that
$\partial f$ is a spin diffeomorphism on components of $\partial N$.
We let $S_{\overline{N}}$ be the standard spinor bundle on $\overline{N}$,
and similarly for $S_{\overline{M}}$.

We can identify
$S_{\overline{N}} \Big|_N$
with $S_N$, and similarly for $S_{\overline{M}}$.

\begin{remark} \label{3.1}
We could phrase what follows just in terms of $N$, $M$ and their normal bundles, but it
seems more illuminating to include the ambient spaces $\overline{N}$ and $\overline{M}$.
\end{remark}

\begin{remark} \label{3.2}
To clarify the relation between the spinor bundle $S_N$ considered here and
that considered in Section \ref{sect2}, we mention some facts about spinors.
Suppose first that $n$ is odd.
Then
$\dim(S_{\R^{n,1}}) = 2^{\frac{n+1}{2}}$
and $\dim(S_{\R^{n,1}} \otimes S_{\R^{n,1}}^*) = 2^{n+1}$, which
equals $\dim(\Lambda^* \R^{n,1})$. 
We can identify
$S_{\R^{n,1}} \otimes S_{\R^{n,1}}^*$ with 
$\Lambda^* \R^{n,1}$.  Using a timelike unit vector $e_0$, 
the latter can be identified with
$\Lambda^* \R^{n}  \widehat{\otimes} \Lambda^* \R^{1}$. Compare with Section \ref{2.6}.

Now suppose that $n$ is even.  
Then $\dim(S_{\R^{n,1}}) = 2^{\frac{n}{2}}$
and $\dim(S_{\R^{n,1}} \otimes S_{\R^{n,1}}^*) = 2^n$. 
On the other hand,
$\dim(\Lambda^* \R^{n,1}) = 2^{n+1}$.  It turns out that one can identify
$S_{\R^{n,1}} \otimes S_{\R^{n,1}}^*$, as a
$\Spin(n,1)$-module, with $\Lambda^*(\R^{n,1})/
(\omega \sim * \omega)$.  Using a timelike unit vector $e_0$, 
the latter can be identified
with $\Lambda^* \R^{n}$. Compare with Section \ref{subsect2.1}.
\end{remark}

The restriction to $N$ of the connection on $S_{\overline{N}}$ has the
local form
\begin{equation} \label{3.3}
  \nabla^W_\sigma =
  e_\sigma + \frac{1}{8} \omega_{\alpha \beta \sigma} [\gamma^\alpha,
    \gamma^\beta] + \frac12 \omega_{0 \alpha \sigma} \gamma^0 \gamma^\alpha,
\end{equation}
where $\alpha$ and $\beta$ are summed from $1$ to $n$. Note that $\nabla^W$
is generally not a unitary connection, because of the last term on the
right-hand side of (\ref{3.3}).
However, the ensuing Dirac-type operator
\begin{equation} \label{3.4}
  D^W = - \sqrt{-1} \sum_{\sigma = 1}^n \gamma^\sigma \nabla^W_{\sigma}
\end{equation}
is formally self-adjoint.  To verify this, one can use the fact that in
normal coordinates around a point of $N$, one has
\begin{equation} \label{3.5}
  [\nabla^W_\sigma, \gamma^\alpha] = \omega_{0 \alpha \sigma} \gamma^0.
  \end{equation}

The restriction of $S_{\overline{M}}$ to $M$ is isomorphic to $S_M$. There
are corresponding connections on $S_M^*$.  As in Section \ref{sect2},
it will be convenient to use the isomorphism of $\Spin(n,1)$-modules
$S_M^* \cong S_M$ to transfer
the connection from $S_M^*$ to $S_M$. There is a subtlety in that the
isomorphism is not unitary.  If $\sigma \in \Spin(n,1)$ is a
transition function for $S_M$ then the corresponding transition function
for $S_M^*$ is $\sigma^{-T}$. This is related to $\sigma$ by $\sigma^{-T} = C \sigma C^{-1}$,
where $C$ is the charge conjugation matrix. In short, 
when written on $S_M$, one finds that 
the connection on $S_M^*$ takes the local form
\begin{equation} \label{3.7}
  \widehat{\nabla}^W_{\widehat{\sigma}} =
  e_{\widehat{\sigma}} + \frac{1}{8} 
\widehat{\omega}_{{\widehat{\alpha}} {\widehat{\beta}} {\widehat{\sigma}}}
  [\widehat{\gamma}^{\widehat{\alpha}},
    \widehat{\gamma}^{\widehat{\beta}}] - \frac12
  \widehat{\omega}_{\widehat{0} \widehat{\alpha} {\widehat{\sigma}}}
  \widehat{\gamma}^{\widehat{0}} \widehat{\gamma}^{\widehat{\alpha}}.
\end{equation}
Note the change in sign in the last term as compared with
(\ref{3.3}).

Let $\nabla^N$ be the connection on
$E = S_N \otimes f^* S_M^*$. It takes the local form
\begin{equation} \label{3.9}
  \nabla^N_\sigma =
    e_\sigma + \frac{1}{8} \omega_{\alpha \beta \sigma} [\gamma^\alpha,
      \gamma^\beta] + \frac12 \omega_{0 \alpha \sigma} \gamma^0 \gamma^\alpha
    + \frac{1}{8}
    \widehat{\omega}_{\widehat{\alpha} \widehat{\beta} \sigma} [
      \widehat{\gamma}^{\widehat{\alpha}},
      \widehat{\gamma}^{\widehat{\beta}}] -
    \frac12 \widehat{\omega}_{\widehat{0} \widehat{\alpha} \sigma}
  \widehat{\gamma}^{\widehat{0}} \widehat{\gamma}^{\widehat{\alpha}}.
\end{equation}
The corresponding Dirac-type operator
\begin{equation} \label{3.10}
  D^N = - \sqrt{-1} \sum_{\sigma = 1}^n \gamma^\sigma \nabla^N_\sigma
\end{equation}
is formally self-adjoint except for the term
$\sqrt{-1} \gamma^\sigma \cdot \frac12
\widehat{\omega}_{\widehat{0} \widehat{\alpha} \sigma}
\widehat{\gamma}^{\widehat{0}} \widehat{\gamma}^{\widehat{\alpha}}$. 

\subsection{Time-symmetric background space} \label{subsect3.1}

Suppose that $M$ is a totally geodesic subspace of $\overline{M}$.
Then $\widehat{\omega}_{\widehat{0} \widehat{\alpha} \sigma} = 0$.
We are in a situation analogous to Section \ref{sect2}, except that $N$ is
now a hypersurface-with-boundary in a Lorentzian manifold. We give the
extensions of results from Section \ref{sect2}.

Let $C$ be a nonempty union of connected components of $\partial N$.
We define the quasilocal energy ${\mathcal E}_C$ as in Definition \ref{2.19}.

\begin{lemma} \label{3.11} If $f : N \rightarrow M$ is an isometric diffeomorphism, and
  $N$ is totally geodesic in $\overline{N}$, then ${\mathcal E}_C = 0$.
  \end{lemma}
\begin{proof}
  The proof is the same as that of Proposition \ref{2.23}.
  \end{proof}

Define $T_{AB} = {R}^{\overline{N}}_{AB} - \frac12
  {R}_{\overline{N}}
g^{\overline{N}}_{AB}$, for $0 \le A,B \le n$.

\begin{proposition} \label{3.12} Suppose that $M$ has nonnegative curvature operator.
  If
\begin{equation} \label{3.13}
  2 \left( T_{00} - \sqrt{- \sum_{\alpha = 1}^n T_{0\alpha} T^{0 \alpha}}
  \right) \ge
  | \Lambda^2 df | (f^* R_M)
\end{equation}
then ${\mathcal E}_C \ge 0$.
\end{proposition}
\begin{proof}
  Using the calculations in \cite{Witten (1981)}, the analog of (\ref{2.4}) is
  \begin{equation} \label{3.14} 
    (D^N)^2 = (\nabla^N)^* \nabla^N + \frac12
    \left( T_{00} + T_{0\alpha} \gamma^0 \gamma^\alpha \right)
    - \frac14 
  [\gamma^\sigma, \gamma^\tau] \left(
  \frac18 \widehat{R}_{\widehat{\alpha} \widehat{\beta} \sigma \tau}
          [\widehat{\gamma}^{\widehat{\alpha}},
  \widehat{\gamma}^{\widehat{\beta}}] \right).
\end{equation}
If $D^N \psi = 0$, we obtain
\begin{align} \label{3.15}
& - \int_{\partial N} \langle \psi, \nabla^E_{e_n} \psi \rangle \: \dvol = 
\int_N | \nabla^N \psi |^2 \: \dvol + \\
& \int_N \langle \psi,
\left( \frac12
    \left( T_{00} + T_{0\alpha} \gamma^0 \gamma^\alpha \right)
    - \frac14 
  [\gamma^\sigma, \gamma^\tau] \left(
  \frac18 \widehat{R}_{\widehat{\alpha} \widehat{\beta} \sigma \tau}
          [\widehat{\gamma}^{\widehat{\alpha}},
  \widehat{\gamma}^{\widehat{\beta}}] \right) \right) 
\psi \rangle \: \dvol. \notag 
\end{align}
Using (\ref{2.9}), the proposition follows.
  \end{proof}

We now give an analog of (\ref{2.30}).
If $D^N \psi = 0$ then the analog of (\ref{2.6}) is
\begin{equation} \label{3.16}
  \nabla^N_{e_n} \psi = D^{\partial N}_{Riem} \psi + 
  \frac{H_{\partial N}}{2} - \frac12 \gamma^n \gamma^i
  \widehat{\gamma}^{\widehat{n}}
  \widehat{\gamma}^{\widehat{j}}  \widehat{A}_{\widehat{j}i} \psi
  - \frac12 \omega_{0ii} \gamma^0 \gamma^n \psi + \frac12
  \omega_{0ni} \gamma^0 \gamma^i \psi,
  \end{equation}
where $D^{\partial N}_{Riem}$ is the intrinsic Riemannian Dirac-type
operator on $\partial N$. 
Then
\begin{align} \label{3.17}
  - \int_{\partial N} \langle \psi, \nabla^E_{e_n} \psi \rangle \: \dvol =  
& - \int_{\partial N} \langle \psi,
  D^{\partial N}_{Riem}
  \psi \rangle \: \dvol_{\partial N} - \frac12
  \int_{\partial N} H_{\partial N} |\psi|^2 \: \dvol_{\partial N} + \\
  & \frac12
  \int_{\partial N} \langle \psi, \gamma^n \gamma^{i}
  \widehat{\gamma}^n \widehat{\gamma}^{\widehat{j}}
  \widehat{A}_{\widehat{j}i} \psi
  \rangle \: \dvol_{\partial N} + \notag \\
  & \frac12
  \int_{\partial N} \langle \psi, \gamma^0 \gamma^{n}
  \omega_{0ii} \psi
  \rangle \: \dvol_{\partial N} - \notag \\
  & \frac12
  \int_{\partial N} \langle \psi, \gamma^0 \gamma^{i}
  \omega_{0ni} \psi
  \rangle \: \dvol_{\partial N}. \notag
\end{align}
If $\psi \in {\mathcal C}_C$ then the last term in (\ref{3.17}) vanishes, as
$T$ anticommutes with
$\gamma^0 \gamma^i$, 

To analyze the next-to-last term in (\ref{3.17}), 
suppose that $n$ is even and $\overline{1}_C = 1_C$. 
Writing $\psi_{min} \Big|_{\partial N} = 1_C +
\tau^n \wedge \phi$ for $\phi \in \Omega^*(\partial N)$, we have
$\gamma^0 \gamma^n \psi_{min} \Big|_{\partial N} = - \gamma^n \gamma^0 (1_C + \tau^n \wedge \phi)$.
The $\Z_2$-grading operator $\epsilon$ of $S_N$, when acting on 
$S_N \otimes S_N^* \cong \Lambda^* T^*N$, is a $4^{th}$ root of $1$
(depending on the congruence class of $n$ modulo $4$) times
\begin{equation} \label{3.18}
\gamma^1 \gamma^2 \ldots \gamma^n =
(-1)^{\frac{n}{2}} (E^1 - I^1) (E^2 - I^2) \ldots (E^n - I^n).
\end{equation}
Applying both sides to $1_C$ gives $\gamma^0 1_C = c  1_C \tau^1 \wedge \ldots \tau^n$ for a
complex constant $c$ of unit norm, depending on $n$.
Then $\gamma^n \gamma^0 1_C = \sqrt{-1} (E^n - I^n) \cdot c 1_C
(\tau^1 \wedge \ldots \tau^n) = c^\prime 1_C \tau^1 \wedge \ldots \wedge
\tau^{n-1}$, for another unit constant $c^\prime$. In particular, the
following terms vanish:
\begin{equation} \label{3.19}
  \langle 1_C, \gamma^0 \gamma^n \omega_{0ii} 1_C \rangle
  =   \langle \tau^n \wedge \phi, \gamma^0 \gamma^n \omega_{0ii} 1_C \rangle
  =  \langle 1_C, \gamma^0 \gamma^n \omega_{0ii} \tau^n \wedge \phi \rangle
  = 0.
\end{equation}
As in (\ref{2.30}), we obtain
\begin{equation} \label{3.20}
  {\mathcal E}_C =  - \frac12
  \int_{C} H_{\partial N} \: \dvol_{\partial N} + \frac12
  \int_{C} (\partial f)^* H_{\partial M} \: \dvol_{\partial N}
+Q(\phi),
\end{equation}
where $Q(\phi)$ is an explicit homogeneous expression of order two in $\phi$. 
There is a similar discussion when $n$ is odd.

If $K$ denotes the second fundamental form of $N$ in $\overline{N}$ then the sum $\omega_{0ii}$ on $\partial N$, which will appear in the next proposition, equals $\Tr(K) - K(e_n, e_n)$.

\begin{proposition} \label{3.21}
Suppose that
\begin{itemize}
\item $M$ is a convex domain in $\R^n$,
\item $T_{00} - \sqrt{- \sum_{\alpha = 1}^n T_{0\alpha} T^{0 \alpha}} \ge 0$,
\item $H_{\partial N} - |\omega_{0ii}| \ge (\partial f)^* H_{\partial M}$, and
\item $T_{00} - \sqrt{- \sum_{\alpha = 1}^n T_{0\alpha} T^{0 \alpha}} > 0$ somewhere
or $H_{\partial N} - |\omega_{0ii}| > (\partial f)^* H_{\partial M}$ somewhere.
\end{itemize}
Then $\Ker({\mathcal D}) = 0$.
\end{proposition}
\begin{proof}
The proof is similar to that of Proposition \ref{2.15}.
\end{proof}

\begin{proposition} \label{3.22}
Whenever $\Ker({\mathcal D}) = 0$, the quasilocal energy can be
described as follows. 
Within the spinor module  associated to $\R^n$,
let $\{\epsilon_a\}_{a=1}^{2^{\frac{n}{2}-1}}$ be an orthonormal 
basis for $\Ker(\gamma^0 \gamma^n - I)$ and let 
$\{\epsilon^\prime_a\}_{a=1}^{2^{\frac{n}{2}-1}}$ be an orthonormal 
basis for $\Ker(\gamma^0 \gamma^n + I)$.
Extending $\epsilon_a$ and $\epsilon^\prime_a$ to constant-valued sections of
$S_M$, let
$\psi_a \in C^\infty(N; S_N)$ be a harmonic spinor on $N$ with boundary 
value in $(\partial f)^* \epsilon_a \cdot 1_C + \Ker(\gamma^0 \gamma^n + I)$ and let
$\psi^\prime_a \in C^\infty(N; S_N)$ be a harmonic spinor on $N$ with boundary 
value in $(\partial f)^* \epsilon^\prime_a \cdot 1_C + \Ker(\gamma^0 \gamma^n - I)$
Then
\begin{equation} \label{3.23}
{\mathcal E}_C = - 2^{- n/2} \sum_{a=1}^{2^{\frac{n}{2} - 1}} \int_{\partial N}
\left( \langle \psi_a, \nabla^N_{e_n} \psi_a \rangle  +
\langle \psi^\prime_a, \nabla^N_{e_n} \psi^\prime_a \rangle
\right) \dvol_{\partial N}.
\end{equation}
In particular, this is true
under the hypotheses of Proposition \ref{3.21}
\end{proposition}
\begin{proof}
The proof is the same as for Proposition \ref{2.51}.
\end{proof}

\subsection{Background space in $\R^{n,1}$} \label{subsect3.2}

We now let the background space be a more general submanifold of Minkowski space.
Let $\overline{M}$ be the Lorentzian $\R^{n,1}$ and suppose that $M$ is a
compact connected spacelike hypersurface-with-boundary in $\overline{M}$. 
In general, the term $\widehat{\omega}_{\widehat{0} \widehat{\alpha} \sigma}$ in
(\ref{3.9}) need not be zero and there is an apparent problem with formal
self-adjointness of $D^N$ on the interior of $N$. We can get around this by changing the inner product.
We use the fact that $S_M$ is the restriction of the trivial bundle
$S_{\overline{M}}$ to $M$.
Choose a constant timelike unit vector field ${\mathcal T}$ in $\overline{M}$.
There is a corresponding inner product $\langle \cdot, \cdot \rangle_{\mathcal T}$
on the trivial bundle
$S_{\overline{M}}$.
Namely, if $(\cdot, \cdot)$ is the indefinite
$\Spin(n,1)$-invariant bilinear form on 
$S_{\overline{M}}$ then the inner product on $S_{\overline{M}}$ is given by 
$\langle s_1, s_2 \rangle_{\mathcal T} = (\sqrt{-1} c({\mathcal T}) s_1, s_2)$,
where $c({\mathcal T})$ is Clifford multiplication by ${\mathcal T}$. 
We can then restrict $(S_{\overline{M}},  \langle \cdot, \cdot \rangle_{\mathcal T})$
to $M$, dualize to get an inner product $\langle \cdot, \cdot \rangle_{sa}$ on
$S_M^*$
 and pullback to $N$, to obtain an inner product
$\langle \cdot, \cdot \rangle_{E,sa}$ on $E = S_N \otimes f^* S^*_M$.

In order to express the boundary conditions, it will be convenient to write
things more intrinsically on $M$.
Let $U \subset M$ be a connected open subset on which an oriented orthonormal frame
$\{\widehat{e}_{\widehat{\alpha}}\}_{\widehat{\alpha} = 1}^n$ is defined and let
$m_0 \in U$ be a basepoint.
Since the
connection on $S_M^*$ is flat,
 we can locally write the expression
$\frac{1}{8} 
\widehat{\omega}_{{\widehat{\alpha}} {\widehat{\beta}} {\widehat{\sigma}}}
  [\widehat{\gamma}^{\widehat{\alpha}},
    \widehat{\gamma}^{\widehat{\beta}}] \otimes \widehat{\tau}^{\widehat{\sigma}} - 
\frac12
  \widehat{\omega}_{\widehat{0} \widehat{\alpha} {\widehat{\sigma}}}
  \widehat{\gamma}^{\widehat{0}} \widehat{\gamma}^{\widehat{\alpha}}
\otimes \widehat{\tau}^{\widehat{\sigma}}
$ in (\ref{3.7}) as
$g^{-1} dg$ for some $g : U \rightarrow \Spin(n,1)$. Let $\rho : \Spin(n,1) \rightarrow \SO(n,1)^+$ be the double cover.
If $\widehat{e}_{\widehat{0}}(m_0)$ is a unit normal vector to $M$ at $m_0$ then
we can partially normalize $g$ by specifying that $g(m_0)$ sends
$\widehat{e}_{\widehat{0}}(m_0)$ to ${\mathcal T}$ in $T_{m_0} \overline{M}$.
The remaining ambiguity in $g$ will be left multiplication
by a constant element of $\Spin(n)$.
On the other hand, a change of oriented orthonormal frame corresponds to
a map $\phi : U \rightarrow \SO(n)$ which lifts, using the spin structure, to a 
a map $\widetilde{\phi} : U \rightarrow \Spin(n)$.  The effect of 
the new frame is to
multiply $g$ on the right by $\widetilde{\phi}^{-1}$. 

If $U_1$ and $U_2$ are overlapping such domains then their corresponding
orthonormal frames will be related by a map $\phi_{12} : U_1 \cap U_2 \rightarrow
\SO(n)$, which we lift using the spin structure
to $\widetilde{\phi}_{12} : U_1 \cap U_2 \rightarrow \Spin(n)$.
Then the maps $g_i : U_i \rightarrow \Spin(n,1)$ are related on
$U_1 \cap U_2$ by 
$g_1 = u_{12} g_2 \widetilde{\phi}_{12}^{-1}$, for some constant element
$u_{12} \in \Spin(n)$. In effect, this uses the fact that 
$\widehat{e}_{\widehat{0}}$
is well-defined on $M$. Similarly if $U_1 \cap U_2$ intersects $\partial M$ then
using the fact that the unit normal 
$\widehat{e}_{\widehat{n}}$ is well-defined,
we can assume that $\widetilde{\phi}_{12} 
\Big|_{U_1 \cap U_2 \cap \partial M}$ takes values in
$\Spin(n-1)$.

On the domain $U$, put $A = g^* g$ and define 
a weighted inner product
$\langle \cdot, \cdot \rangle_{S_M^*,sa}$
by $\langle s_1, s_2 \rangle_{S_M^*,sa} = 
\langle s_1, A s_2 \rangle_{S_M^*}$ for
$s_1, s_2 \in S_M^*$.  Left multiplication of
$g$ by an element of $\Spin(n)$ doesn't change $A$, so $A$ is independent of
choices. 
The connection (\ref{3.7}) on $S_M^*$ is compatible with 
$\langle \cdot, \cdot \rangle_{S_M^*,sa}$, since the
ambient flat connection on $S_{\overline{M}}$ is compatible with
$\langle \cdot, \cdot \rangle_{\mathcal T}$.

Given a smooth spin map $f : (N, \partial N) \rightarrow (M, \partial M)$
which is an isometric spin diffeomorphism on each connected component of
$\partial N$,
we construct the Clifford module $E = S_N \otimes f^* S_M^*$ on $N$, with the
inner product
$\langle \cdot, \cdot \rangle_{E,sa} =
\langle \cdot, \cdot \rangle \Big|_{S_N} \otimes
f^* \langle \cdot, \cdot \rangle \Big|_{S_M^*,sa}$
The analog of (\ref{2.3}) is
\begin{align} \label{3.24}
& \int_N \langle D^N \psi_1, \psi_2 \rangle_{E,sa} \dvol_N -
\int_N \langle \psi_1, D^N \psi_2 \rangle_{E,sa} \dvol_N  = \\
& - \sqrt{-1} \int_{\partial N} \langle \psi_1, \gamma^n \psi_2 \rangle_{E,sa} 
\dvol_{\partial N}. \notag
\end{align}

\begin{proposition} \label{3.25}
If $N$ satisfies the dominant energy condition
$T_{00} \ge \sqrt{- \sum_{\alpha = 1}^n T_{0\alpha} T^{0 \alpha}}$ then
for any $\psi \in C^\infty(N; E)$ satisfying $D^N \psi = 0$, we have
\begin{equation} \label{3.26}
- \int_{\partial N} \langle \psi, \nabla^N_{e_n} \psi \rangle_{E,sa}
 \: \dvol_{\partial N}
\ge 0.
\end{equation}
\end{proposition}
\begin{proof}
On the interior of $N$, we can think of $D^N$ as
$D^W \otimes \Id_{\C^{2^{n/2}}}$. By the calculations in \cite{Witten (1981)},
in this representation we can write
\begin{equation} \label{3.27}
(D^N)^2 = 
\left( 
(\nabla^N)^*_{sa} \nabla^N + \frac12 (T_{00} + T_{0 \alpha} \gamma^0 \gamma^\alpha)
\right) \otimes \Id_{\C^{2^{n/2}}}.
\end{equation}
As $D^N \psi = 0$, after integrating by parts we obtain from (\ref{3.27}) that
\begin{align} \label{3.28}
& - \int_{\partial N} \langle \psi, \nabla^N_{e_n} \psi \rangle_{E,sa} 
\: \dvol_{\partial N} = \\
& \int_N \left( \langle \nabla \psi, \nabla \psi
\rangle_{E,sa} +
\frac12 \langle \psi, 
(T_{00} + T_{0 \alpha} \gamma^0 \gamma^\alpha) \psi \rangle_{E,sa} \right) 
\dvol. \notag
\end{align}
The proposition follows.
\end{proof}

To follow what was done in Section \ref{sect2}, we put $T = \gamma^0 \gamma^n \widehat{\gamma}^{\widehat{0}}
\widehat{\gamma}^{\widehat{n}}$, acting on $E \Big|_{\partial N}$.
We again have $T1=1$.
Note that $T$ 
may not be self-adjoint with respect to 
$\langle \cdot, \cdot \rangle_{E,sa}$.

Given $\sigma \in \Omega^*(\partial N)$ with $T \sigma = \sigma$,
there are two natural boundary conditions to impose on the equation
$D^N \psi = 0$; we could ask that $\psi \Big|_{\partial N} \in \sigma +
\Ker(T+I)$ or we could ask that 
$\psi \Big|_{\partial N} \in \sigma +
(\Ker(T-I))^\perp$. For concreteness, we take the second choice.

Consider the operator $D^N$ acting on elements
$\psi \in C^\infty(N; E)$ that satisfy $\psi \Big|_{\partial N} \in
(\Ker(T-I))^\perp$. This defines an elliptic boundary condition.
Let ${\mathcal D}$ be the corresponding operator, densely defined on
$C^\infty(N; E)$.
Here ${\mathcal D}$ may not be self-adjoint but it still 
has discrete spectrum with eigenspaces of finite multiplicity.

We wish to find some
$\psi \in C^\infty(N; E)$ so that
  $D^N \psi = 0$ and $\psi \Big|_{\partial N} \in  \sigma + (\Ker(T-I))^\perp$.
This is an elliptic boundary value problem.
Suppose that we can do this. Note that 
$\gamma^n \psi \Big|_{\partial N} \in  \gamma^n \sigma + (\Ker(T+I))^\perp$
 Put
\begin{equation} \label{3.29}
{\mathcal H} =
\left\{ \: \eta \in C^\infty(N;E) \: : \: D^N \eta = 0, 
\eta \Big|_{\partial N} \in \Ker(T+I) \right\}.
\end{equation} 
If $\eta \in {\mathcal H}$ then 
from (\ref{3.24}), we see that
$\int_{\partial N} \langle \sigma, \gamma^n \eta \rangle \: \dvol_{\partial N} = 0$.
Conversely, given $\sigma \in \Omega^*(\partial N)$ with 
$T\sigma = \sigma$, Fredholm theory implies that if 
$\int_{\partial N} \langle \sigma, \gamma^n \eta \rangle \: \dvol_{\partial N} = 0$
for all $\eta \in {\mathcal H}$ then there is some
$\psi \in C^\infty(N; E)$ so that
  $D^N \psi = 0$ and $\psi \Big|_{\partial N} \in
\sigma + (\Ker(T-I))^\perp$; c.f. \cite[Theorem 2.4.5]{Agranovich (1997)}. The section
$\psi$ is unique up to addition by elements of $\Ker({\mathcal D})$.

As before, $C$ is a nonempty union of connected components of $\partial N$.
We would like to take $\sigma = 1_C$. To be sure that $\psi$ exists, we may
have to slightly modify this choice.  Let 
$P_{\gamma^n {\mathcal H}}$ be 
orthogonal projection
from $\Omega^*(\partial N)$ to $\gamma^n {\mathcal H}$.
We put $\overline{1}_C = 1_C - 
P_{\gamma^n {\mathcal H}} 1_C$, take $\sigma = \overline{1}_C$
and construct $\psi$ accordingly.

\begin{definition} \label{3.30}
Put 
\begin{equation} \label{3.31}
{\mathcal C}_C = \left\{ \psi \in C^\infty(N; E) \: : \: D^N \psi = 0,
\psi \Big|_{\partial N} \in
\overline{1}_C + (\Ker(T-I))^\perp \right\}.
\end{equation}
  The quasilocal energy is
\begin{equation} \label{3.32}
{\mathcal E}_C = \inf_{\psi \in {\mathcal C}} \left( - \int_{\partial N} \langle \psi,
  \nabla^N_{e_n} \psi \rangle_{E,sa} \: \dvol_{\partial N} \right).
\end{equation}
\end{definition}

\begin{proposition} \label{3.33}
If $\overline{N} = \overline{M} = \R^{n,1}$ and $N = M$ then ${\mathcal E}_{\partial N} = 0$.
\end{proposition}
\begin{proof}
In this case, the connection forms satisfy $\widehat{\omega} = \omega$. 
Suppose that $n$ is odd.
We can identify $E$ with $\Lambda^*(T^*N) + \tau^0 \wedge \Lambda^*(T^*N)$.
Putting $\gamma^\alpha = \sqrt{-1} (E^\alpha - I^\alpha)$, 
$\widehat{\gamma}^{\widehat{\alpha}} = E^\alpha + I^\alpha$,
$\gamma^0 = \sqrt{-1} (E^0 + I^0)$ and
$\widehat{\gamma}^{\widehat{0}} = - (E^0 - I^0)$, when acting on
$\Omega^*(N) \oplus \tau^0 \wedge \Omega^*(N)$ one can check that
\begin{equation} \label{3.34}
\nabla^N_\sigma = e_\sigma + \omega_{\alpha \beta \sigma} E^\alpha I^\beta 
+ \omega_{0 \alpha \sigma} (E^0 I^\alpha + E^\alpha I^0).
\end{equation}
In particular, $\nabla^N_\sigma 1 = 0$, so $D^N 1 = 0$. 
It follows that $\overline{1} = 1$, so
$1 \in {\mathcal C}$, with $- \int_{\partial N} \langle 1, 
\nabla^N_{e_n} 1 \rangle_{E,sa} \: \dvol_{\partial N} = 0$.
Proposition \ref{3.25} now implies that ${\mathcal E}_{\partial N} = 0$.
The proof when $n$ is even is similar.
\end{proof}

For an alternative approach that maintains self-adjointness, if $U \subset M$ is an
open set as before then on $U \cap \partial M$, we put
\begin{equation} \label{3.35}
T_{sa} = \left[ (\partial f)^*(g^*g) \right]^{- \frac12} \gamma^0 \gamma^n \widehat{\gamma}^{\widehat{0}}
\widehat{\gamma}^{\widehat{n}} \left[ (\partial f)^*(g^*g) \right]^{\frac12},
\end{equation}
which is now a self-adjoint 
idempotent with respect to
$\langle \cdot, \cdot \rangle_{E,sa}$.
To see that this is globally defined, as mentioned before
on an overlap $U_1 \cap U_2 \cap \partial M$ we can assume that
$\widetilde{\phi}_{12}$ commutes with 
$\widehat{\gamma}^{\widehat{0}}$ and
$\widehat{\gamma}^{\widehat{n}}$. 
As 
\begin{equation} \label{3.36}
((u_{12} g \widetilde{\phi}_{12}^{-1})^* u_{12} g 
\widetilde{\phi}_{12}^{-1} )^{\frac12} =
\widetilde{\phi}_{12} (g^* g)^{\frac12} \widetilde{\phi}_{12}^{-1},
\end{equation}
we obtain
\begin{align} \label{3.37}
& \left[(\partial f)^*(g_1^*g_1)\right]^{- \frac12} \gamma^0 \gamma^n \widehat{\gamma}^{\widehat{0}}
\widehat{\gamma}^{\widehat{n}} \left[(\partial f)^*(g_1^*g_1)\right]^{\frac12} = \\
& (\partial f)^* \widetilde{\phi}_{12}
\left\{
\left[(\partial f)^*(g_2^*g_2)\right]^{- \frac12} \gamma^0 \gamma^n \widehat{\gamma}^{\widehat{0}}
\widehat{\gamma}^{\widehat{n}} \left[(\partial f)^*(g_2^*g_2)\right]^{\frac12} \right\}
(\partial f)^* \widetilde{\phi}_{12}^{-1}. \notag
\end{align}
Taking into account that $\widetilde{\phi}_{12}$ encodes how spinors
transform, 
it follows that $T_{sa}$ is globally defined on $\partial N$.
We still have
$T_{sa} \gamma^n + \gamma^n T_{sa} = 0$.

\begin{lemma} \label{3.38}
Consider the operator $D^N$ on elements $\psi \in
C^\infty(N; E)$ satisfying the boundary condition
\begin{equation} \label{3.39}
(T_{sa} + I) \psi \Big|_{\partial N}  = 0.
\end{equation}
It is
formally self-adjoint.
\end{lemma}
\begin{proof}
If $\psi_1, \psi_2 \in
C^\infty(N; E)$ satisfy
$(T_{sa}+ I) \psi_1 \Big|_{\partial N} = 
(T_{sa} + I) \psi_2 \Big|_{\partial N} = 0$
then 
\begin{align} \label{3.40}
\int_{\partial N} \langle \psi_1, \gamma^n \psi_2 \rangle_{E,sa} \dvol_{\partial N}
&  =
- \int_{\partial N} \langle \psi_1, \gamma^n T_{sa} \psi_2 \rangle_{E,sa} \dvol_{\partial N} \\
&  =
\int_{\partial N} \langle \psi_1, T_{sa} \gamma^n \psi_2 \rangle_{E,sa} \dvol_{\partial N} \notag \\
& =
\int_{\partial N}  \langle T_{sa} \psi_1, \gamma^n \psi_2 \rangle_{E,sa} \dvol_{\partial N} \notag \\
& =
- \int_{\partial N} \langle \psi_1, \gamma^n \psi_2 \rangle_{E,sa} \dvol_{\partial N},
\notag
\end{align}
so $\int_{\partial N} \langle \psi_1, \gamma^n \psi_2 \rangle_{E,sa} \dvol_{\partial N}$
vanishes.
Using (\ref{3.24}), the lemma follows.
\end{proof}

With reference to Lemma \ref{3.38}, let ${\mathcal D}$ be the self-adjoint
extension.  It has compact resolvent.

To find the right boundary value problem, we want a section $\sigma$ of
$E \Big|_{\partial N}$ so that $T_{sa}\sigma = \sigma$.  Identifying
$E \Big|_{\partial N}$ with $\End \left( S_{N} \Big|_{\partial N} \right)$, an initial
attempt would be the operator
$s_C = 2^{- n/2} 1_C \left[(\partial f)^* (g^* g)\right]^{- \frac12}$ when acting
on $S_{N} \Big|_{\partial N}$. We note that $(g^* g)^{- \frac12}$ is
an isometry from $(S_M^*, \langle \cdot, \cdot \rangle)$ to
$(S_M^*, \langle \cdot, \cdot \rangle_{sa})$.

More explicitly,
at a point $n \in \partial N$, let $\{v_a\}_{a=1}^{2^{\frac{n}{2}}}$ be
an orthonormal basis of $S_N$ and 
let $v_a^*$ denote inner product with respect to $v_a$.
Then we put
\begin{equation} \label{3.41}
s_C  = 2^{- n/2} 1_C (\partial f)^* (g^* g)^{- \frac12} \cdot \Id =
2^{- n/2} 1_C \sum_{a=1}^{2^{\frac{n}{2}}}
v_a \otimes (\partial f)^* (g^* g)^{- \frac12} v_a^*.
\end{equation}
If we divide $\{v_a\}_{a=1}^{2^{\frac{n}{2}}}$ into orthonormal bases for the
$\pm 1$-eigenspaces of $\gamma^0 \gamma^n$ then one sees that 
$T_{sa} s_C = s_C$. 

Let $\pi_{\pm}$ denote orthogonal projections
onto the $\pm 1$-eigenspaces of $T_{sa}$
acting on sections of $E \Big|_{\partial N}$. As before, we may have to slightly
modify the boundary data $s_C$ to ensure that we can solve the boundary value problem.
Put $\overline{s}_C = s_C - P_{\gamma^n \Ker(\mathcal D) \Big|_{\partial N}} s_C$.
Then there is some
$\psi \in C^\infty(N; E)$ so that
$D^N \psi = 0$ and 
$
\pi_+ \left( \psi \Big|_{\partial N} \right)
  = \overline{s}_C$.
The section $\psi$ 
is unique up to addition by elements of
$\Ker({\mathcal D})$. 
We define
${\mathcal E}^{sa}_C$ as in Definition \ref{2.19}, using
$\langle \cdot, \cdot \rangle_{E,sa}$.

\begin{proposition} \label{3.42}
If $N$ satisfies the dominant energy condition then ${\mathcal E}_C \ge 0$
and ${\mathcal E}^{sa}_C \ge 0$.
\end{proposition}
\begin{proof}
This follows from Proposition \ref{3.25}.
\end{proof}

\begin{proposition} \label{3.43}
${\mathcal E}_C$ and ${\mathcal E}^{sa}_C$ only depend on $f$ through its boundary
value $\partial f$.
\end{proposition}
\begin{proof}
The proof is the same as for Proposition \ref{2.49}.
\end{proof}

The quasilocal energy ${\mathcal E}_C$ depends on the
choice of constant timelike unit vector field ${\mathcal T}$ in 
$\R^{n,1}$.  By minimizing over such choices, we obtain the quasilocal mass.

\begin{definition} \label{3.44}
The quasilocal mass is ${\mathcal M}_C = \inf_{{\mathcal T}} {\mathcal E}_C$.
\end{definition}

\subsection{Background space in a product manifold} \label{subsect3.3}

In Section \ref{3.2}, when the background space was a subspace of
$\R^{n,1}$, we effectively performed a gauge transformation in order to
make $\widehat{\nabla}^W$ locally trivial.
This had the effect of making the Dirac-type operator
$D^N$ formally self-adjoint on the interior of $N$. 
Looking at the connection (\ref{3.7}), one sees that what one
really needs is a gauge transformation to make 
$\widehat{\nabla}^W$ unitary.  This
can be achieved if the connection on the ambient
spinor bundle $S_{\overline{M}}$ is unitary.  The easiest situation in which
this is guaranteed is when the Lorentzian manifold
$\overline{M}$ is an isometric product 
$\R \times X$ for some $n$-dimensional Riemannian spin manifold $X$.
Of course, this includes the case when $\overline{M} = \R^{n,1}$. 

Let $M$ be a spacelike hypersurface with boundary in $\overline{M} = \R \times X$.
As in Section \ref{3.2}, we can find local maps $g : U \rightarrow
\Spin(n)$ on $M$ so that 
\begin{equation} \label{3.45}
g \left( \frac{1}{8} 
\widehat{\omega}_{{\widehat{\alpha}} {\widehat{\beta}} {\widehat{\sigma}}}
  [\widehat{\gamma}^{\widehat{\alpha}},
    \widehat{\gamma}^{\widehat{\beta}}] \otimes \widehat{\tau}^{\widehat{\sigma}} - 
\frac12
  \widehat{\omega}_{\widehat{0} \widehat{\alpha} {\widehat{\sigma}}}
  \widehat{\gamma}^{\widehat{0}} \widehat{\gamma}^{\widehat{\alpha}}
\otimes \widehat{\tau}^{\widehat{\sigma}} \right) g^{-1} + g \: dg^{-1}
\end{equation} is a
$1$-form with values in skew-Hermitian matrices 
(as opposed to its vanishing in Section \ref{3.2}). 
We can assume that $g(m_0)$ sends $\widehat{e}_{\widehat{0}}(m_0)$ to the
timelike unit vector $\partial_t$ in $T_{m_0} \overline{M}$.
On an overlap $U_1 \cap U_2$,
the maps $g_1$ and $g_2$ are related by $g_1 = u_{12} g_2 
\widetilde{\phi}_{12}^{-1}$ for maps $u_{12}, \widetilde{\phi}_{12} :
U_1 \cap U_2 \rightarrow \Spin(n)$, where $u_{12}$ is covariantly
constant and $\widetilde{\phi}_{12}$ comes from the change of local
orthonormal frame.

On the domain $U \subset M$, we put $A = g^* g$ and define an inner
product $\langle \cdot, \cdot \rangle_{S^*_M,sa}$ by
$\langle s_1, s_2 \rangle_{S^*_M,sa} = \langle s_1, A s_2 \rangle_{S^*_M}$.

\begin{proposition} \label{3.46}
Suppose that $X$ has nonnegative curvature operator.  If
$\psi \in C^\infty(N; E)$ satisfies $D^N \psi = 0$ and
\begin{equation} \label{3.47}
  2 \left( T_{00} - \sqrt{- \sum_{\alpha = 1}^n T_{0\alpha} T^{0 \alpha}}
  \right) \ge
  | \Lambda^2 df | (f^* R_X)
\end{equation}
then $- \int_{\partial N} 
\langle \psi, \nabla^E_{e_n} \psi \rangle \ge 0$.
\end{proposition}
\begin{proof}
After making local gauge transformations, the proof is the same as that of
Proposition \ref{3.25}. 
\end{proof}

We can define ${\mathcal E}_C$ and ${\mathcal E}^{sa}_C$ as in Section \ref{3.2}.
The analogs of Propositions \ref{3.42} and \ref{3.43} hold.

\end{document}